\UseRawInputEncoding
\documentclass[12pt, leqno]{article}
\usepackage{latexsym}
\usepackage{enumerate,amsmath,amsthm,amssymb,amsfonts,amscd}

\newtheorem{defn}{Definition}[section]
\newtheorem{thm}[defn]{Theorem}
\newtheorem{cor}[defn]{Corollary}

\newtheorem{lem}[defn]{Lemma}

\numberwithin{equation}{section}

\def\F{\mbox{${\cal F}$}}

\newcommand{\cc}{\mathbb C}
\newcommand{\zz}{\mathbb Z}
\newcommand{\nn}{\mathbb N}

\input{arrow.tex}

\begin{document}
\title{\bf On an affirmative solution to Michael's problem, with applications to automatic continuity theory}

\author{S. R. PATEL}

\maketitle \vspace{1.5cm}

\noindent Department of Mathematics, C. U. Shah University, Wadhwan City,\\
Gujarat, INDIA\\
E-mail: srpatel.math@gmail.com
\begin{tabbing}
\> Secondary: 13F25, 32A10,\\
\> 46A32, 46E25, 46H10, 46H40
\end{tabbing}
\newpage
\noindent {\bf Abstract} In 1952, Michael posed a question about
functional continuity of commutative Fr\'{e}chet algebras in his
memoir, known as Michael's problem in the literature. We settle
this in the affirmative along with its equivalent forms. Indeed,
we continue our previous works to directly attack these problems.
We show that the test case for this problem -- the Fr\'{e}chet
algebra $\cal U$ of all entire functions on $\ell^\infty$ -- is,
in fact, a Fr\'{e}chet algebra $\cc[[X]]$, if there exists a
discontinuous character on $\cal U$. An important tool is a
topological version of the tensor (symmetric) algebra over a
Banach space. The elementary, but crucial, idea is to express the
test algebra $\cal U$ as a {\it weighted Fr\'{e}chet symmetric
algebra} over the Banach space $\ell^1$. We summarize the effects
of our affirmative solution in terms of important applications in
automatic continuity theory, including affirmative answers to two
significant, long-standing problems from 1966 and 1978.
\newpage
\section{Introduction and statement of the main theorem}
{\bf 1.1. History of Michael's problem.} An important subject of
study in the theory of Fr\'{e}chet algebras is certain questions
from automatic continuity theory. Viz., we may consider the
(non-)uniqueness of the Fr\'{e}chet algebra topology (briefly: the
Fr\'{e}chet topology in the sequel) on certain Fr\'{e}chet
algebras, (dis)continuity property of certain homomorphisms
between certain Fr\'{e}chet algebras as well as that of (higher
point) derivations on Fr\'{e}chet algebras. This theory has
applications to commutative rings and algebras (e.g.,
Singer-Wermer conjecture (1955) in Fr\'{e}chet algebras ~\cite{37,
35}) and to the theory of SCV (e.g., Corona problem (1962)
~\cite{30} and Gleason's problem (1964) ~\cite{36}).

In 1952, Michael posed a question about functional continuity of
commutative Fr\'{e}chet algebras ~\cite{29}. This problem has a
very strong historical background. A number of significant
analysts worked on this problem since 1952. As consequences, there
are affirmative solutions for special classes of Fr\'{e}chet
algebras (e.g., Arens ~\cite{5}, Shah ~\cite{39}, Ephraim
~\cite{19}, \.Zelazko ~\cite{43}, Akkar, Oubbi and Oudadess
~\cite{2}). Some analysts discussed various test cases (e.g.,
Clayton ~\cite{9}, Dixon and Esterle ~\cite{17}, Dales, author and
Read ~\cite{15}), or discussed various approaches (e.g., Dixon and
Esterle ~\cite{17}, Esterle ~\cite{20}, Dales, author and Read
~\cite{15}). Other analysts discussed various equivalent forms
(e.g., Craw~\cite{10}, Dixon and Fremlin ~\cite{18},
Akkar~\cite{1}).

Despite a lot of efforts by analysts to solve Michael's problem,
it seems that only six significant ideas appeared in the
literature since 1952 ~\cite{5, 39, 9, 17, 20, 15}. The strong
partial result was obtained by Arens in 1958 ~\cite{5}. He showed
that each finitely (resp., rationally) generated, commutative
Fr\'{e}chet algebra is functionally continuous. As far as we know,
the most latest effort was made by Dales, author and Read in 2010.
We showed that a well known test case for this
problem---Fr\'{e}chet algebra $\cal U$ of all entire functions on
$\ell^\infty$ ~\cite{9, 17, 20}---is, in fact, a Fr\'{e}chet
algebra of power series (briefly: FrAPS) ~\cite{15}. So, the
natural question of whether every character (i.e., non-zero
complex homomorphism) is automatically continuous on a Fr\'{e}chet
algebra which is a FrAPS, is Michael's problem itself in disguise.
Incidentally, author asked \.Zelazko in 2004 whether Michael's
problem has an affirmative solution for {\it all} FrAPS [Personal
Communication]. Obviously, an affirmative answer to this question
would extend the Arens' result within the class of FrAPS, thereby
solving Michael's problem in view of ~\cite[Theorem 10.1
(ii)]{15}.\\

\noindent {\bf 1.2. Our past work as a background.} So far, we
have established the uniqueness of the Fr\'{e}chet topology for
{\it all} FrAPS in one indeterminate ~\cite{31} and in finitely
many indeterminates ~\cite{33}. Dales, author and Read have also
obtained an affirmative solution (in a stronger form) to
Dales-McClure's problem (1977) among other interesting automatic
continuity results, including a reduction of Michael's problem as
discussed above. We have used discontinuity of derivations to give
other inequivalent Fr\'{e}chet topology to certain Fr\'{e}chet
algebras, solving Loy's problem (1974) in the Fr\'{e}chet case
~\cite{34}. Also, by using the Read's method, we have constructed
two maiden examples of Fr\'{e}chet algebras ($\F_\infty
\,\textrm{and}\, \F_\infty\,\oplus\,\F_\infty$) with countably
many mutually inequivalent Fr\'{e}chet topologies, to discuss
various avatars of the Singer-Wermer conjecture in Fr\'{e}chet
algebras ~\cite{37, 35}. Such examples are not known even in the
Banach algebra case. As we shall see below, {\it all} these works
have turned out to be stepping stones for the complete solution to
Michael's problem (e.g., see ~\cite[Question, p. 135]{31} and
~\cite[Theorem 11.2]{15}).

Since 1958, it was not clear to functional analysts how to extend
the Arens' result from the finitely generated case to the
countably generated case, in order to obtain functional continuity
of the test algebra $\cal U$. One strongly believed that there
must have some ways to apply appropriate methods within automatic
continuity theory as the problem falls in this theory. So, we
represent the first approach in this paper; in the forthcoming
paper, we exclusively discuss the second approach for the
non-commutative case of Michael's problem. In both the approaches,
an important tool is a topological version of the tensor
(symmetric) algebra over a Banach space. The elementary, but
crucial, idea is to express the test algebra $\cal U$ as a {\it
weighted Fr\'{e}chet symmetric algebra} over the Banach space
$\ell^1$ ~\cite{15, 34}.\\

\noindent {\bf 1.3. Preliminaries.} Throughout the paper,
``algebra" will mean a non-zero, complex commutative algebra with
identity unless otherwise specified. A topological algebra $A$ is
functionally continuous if every character on $A$ is continuous,
and $A$ is functionally bounded if every character on $A$ is
bounded on bounded sets. A {\it Fr\'{e}chet algebra} is a
complete, metrizable locally convex algebra $A$ whose topology
$\tau$ may be defined by an increasing sequence $(p_m)_{m\in \nn}$
of submultiplicative seminorms. The basic theory of Fr\'{e}chet
algebras was introduced in ~\cite{29, 23}. If $A$ is a Fr\'{e}chet
algebra, then it is functionally continuous if and only if it is
functionally bounded ~\cite[Theorem 1]{18}. However, there are
commutative, complete LMC-algebras which are functionally bounded,
but not functionally continuous; e.g., see ~\cite[pp. 52-53]{29}.

We write $\F_\infty$ for the algebra $\cc[[X_1, X_2, \dots]]$ of
all formal power series in countably many commuting indeterminates
$X_1, X_2, \dots$, with complex coefficients. We shall also
require the non-commutative version of $\F_\infty$, denoted by
$\cal B$ = $\cc_{nc}[[X_1, X_2, \dots]]$. A fuller description of
these algebras is given in ~\cite[\S 9]{15}; see also ~\cite{20}
for $\F_\infty$. The algebra $\F_\infty$ is a Fr\'{e}chet algebra
under the usual topology $\tau_c$ of coordinatewise convergence
defined by the coordinate projections $\pi_I, I \in
(\zz^+)^{<\omega}$. The defining sequence $(p_m^{'})$ of seminorms
is: $$p_m^{'}(\sum_{J \in (\zz^+)^{<\omega}}\lambda_J X^J) =
\sum\{\mid \lambda_J \mid : J \in (\zz^+)^m, \mid J\mid \leq
m\}\;(m \in \nn).$$ Clearly, each $p_m^{'}$ is a proper seminorm
on $\F_\infty$.

Some remarks on $\F_\infty$ are in order. The algebra $\F_\infty$
(resp., $\cal B$) is a graded algebra. In ~\cite{20}, Esterle
remarks that $\F_\infty$ is a local algebra with its unique
maximal ideal $\textrm{ker}\,\phi_0$, $\phi_0$ the unique
character of $\F_\infty$. Also, $\F_\infty$ is an integral domain,
with all its principal ideals closed, but that he does not know
whether or not all finitely generated ideals in $\F_\infty$ are
closed. Kopp shows that $\F_\infty$ is a Fr\'{e}chet algebra of
finite type and so, all finitely generated ideals are closed in
$\F_\infty$ ~\cite[Remark, p. 222]{25}. It was shown that there is
a continuous, injective homomorphism from $\F_\infty$ into $\F_2$
as an extension of ~\cite[Theorem 2.2]{15} (see ~\cite[Theorem
9.1]{15}). Also, the algebra $\F_k$ cannot be embedded
algebraically in $\F_1 = \F$ for each $k \geq 2$ ~\cite[Theorem
2.6]{15} (see ~\cite[Theorem 11.8]{15} for the second proof using
some topological arguments). The similar arguments for $\F_\infty$
allows us to establish the following
\begin{thm} \label{Thm 2.6_DPR} There is neither topological nor algebraic embedding of
$\F_\infty$ into $\F$.
\end{thm}
\noindent {\it Proof.} Assume towards a contradiction that $\theta
: \F_\infty \rightarrow \F$ is such a topological embedding. Then
$\theta$ is not a surjection, for this would imply that $\F_\infty
\cong \F$ topologically, and this is impossible because $\F$ has a
unique Fr\'{e}chet topology ~\cite[Example 1, p. 276]{4} whereas
$\F_\infty$ has two inequivalent Fr\'{e}chet topology ~\cite{37}.
Hence $\theta$ is a proper embedding, and so we may regard $A :=
\theta(\F_\infty)$ as a Fr\'{e}chet subalgebra of $\F$. In fact,
since $\F_\infty$ is a Fr\'{e}cht algebra of finite type and
$\theta$ is continuous, by ~\cite[Theorem 1]{25}, $A$ is a closed
subalgebra of $\F$. Since $\textrm{ker}\,\phi_0$ is a unique
closed maximal ideal in $\F_\infty$, its image
$\theta(\textrm{ker}\,\phi_0)$ is also a unique closed maximal
ideal in $A$ ($A$ being local as $\F_\infty$ is). Further
$\pi_0|_A$ is a continuous character on $A$, and so $AX$ is also a
closed maximal ideal in $A$. On the other hand, by ~\cite[p.
218]{25}, $\textrm{ker}\,\phi_0$ is a radical Fr\'{e}chet algebra
of finite type, and so its image $\theta(\textrm{ker}\,\phi_0)$ is
closed in $\F$. Since $A$ is a closed subalgebra of $\F$,
$\theta(\textrm{ker}\,\phi_0)$ is also closed in $A$. Since $A$ is
a local algebra, we have $AX = \theta(\textrm{ker}\,\phi_0)$. Thus
$A$ is, indeed, a FrAPS in $\F$ since it is a closed subalgebra
containing $X$. We may assume, w.l.o.g., that $\theta(X_1) = X$.

By ~\cite[Theorem 3.3]{31}, the topology of $A$ is given by a
sequence of norms, say by the sequence $(\|\cdot\|_n)$. By
~\cite[Corollary 4.2]{31}, $A$ has a unique Fr\'{e}chet topology,
and so the topology given by the sequence $(\|\cdot\|_n)$ is
equivalent to the usual topology $\tau_c$, given by the sequence
$(p_m^{'})$ of proper seminorms, on $A$. In particular, there
exist $m \in \nn$ and $C > 0$ such that $\|f\|_1 \leq Cp_m{'}(f)$
for all $f \in \F_\infty$. But now $$\|X^{m + 1}\|_1 \leq
Cp_m{'}(X^{m + 1}) = 0,$$ a contradiction of the fact that
$\|\cdot\|_1$ is a norm on $\F_\infty$. Thus there is no such
topological embedding.

Next, assume towards a contradiction that $\theta : \F_\infty
\rightarrow \F$ is such an algebraic embedding. Then $\theta$ is
not a surjection, for this would imply that $\F_\infty \cong \F$
algebraically, and this is impossible because $\F$ is noetherian,
but $\F_\infty$ is not. Hence $\theta$ is a proper embedding. By
~\cite[Theorem 11.2]{15}, $\theta$ is continuous, and so we may
regard $A := \theta(\F_\infty)$ as a Fr\'{e}chet subalgebra of
$\F$. Now, following the arguments given in the above two
paragraphs, we lead to a contradiction. Thus there is no such
algebraic embedding. $\hfill \Box$

We shall return to an interesting fact about this theorem in view
of ~\cite[Theorems 9.1 and 10.1]{15} at the end of the paper.\\

\noindent {\bf 1.4. Main theorem, with comments on the
strategies.} In this paper, we shall be concerned with the
affirmative solutions to Michael's two problems. We shall briefly
discuss some general theory of topological tensor (symmetric)
algebras over a Banach space in \S 2, in order to establish
notation that will be used throughout the paper. In particular, we
mainly study certain semigroup (Banach/Fr\'{e}chet) algebras over
the semigroup $S = (\zz^+)^{<\omega}$, which are graded
subalgebras of $\F_\infty$ ~\cite[\S 9, \S 10]{15}. These algebras
are required for our approach to solve Michael's problem in \S 3.
Thus these algebras over the Banach space $\ell^1$ are the
technical main-spring of the paper.

In our approach, the essential ingredient is as follows. First, we
describe the test algebra $\cal U$ in terms of weighted
Fr\'{e}chet symmetric algebra $\widehat{\bigvee}_WE$ over the
Banach space $E\,=\,\ell^1$ (this example, discussed in
~\cite{34}, motivated the ideas of the present paper). Then we
show that the existence of a discontinuous character $\phi$ on
$\cal U$ leads us to produce a totally discontinuous higher point
derivation $(d_n)_{n\,\in\,\zz^+}$ on $\cal U$ at $\phi$. By
~\cite[Theorems 11.2 and 10.1 (ii)]{15}, this would imply that
${\cal U}\,=\,\F$, a contradiction. (Discussions along these lines
always skirt Dales-McClure's problem, solved affirmatively by
Dales, author and Read in 2010.) We give some crucial remarks on
our approach, along with an alternative proof of the following\\

\noindent {\bf Main Theorem.} {\it All characters on the
commutative Fr\'{e}chet algebra $({\cal U}, (q_m))$ are
continuous. In particular, these characters are, indeed, the point
evaluation mappings, that is, if $\phi$ is a character on $\cal
U$, then there exists $z\,\in\,\ell^\infty$ such that
$\phi\,=\phi_z$, where $\phi_z(f)\,=\,f(z)$ for all $f\,\in\,\cal
U$.}

As consequences, we have affirmative solutions to Michael's two
problems as follows (see Corollary 3.2 and Corollary 3.3 below).

\noindent {\bf Statements.} {\it 1. Every commutative Fr\'{e}chet
algebra is functionally continuous. 2. Every commutative, complete
LMC-algebra is functionally bounded.}

We extend certain results about functional continuity of Stein
algebras, established by Forster ~\cite{22} and Ephraim
~\cite{19}. We also discuss implications of our solution in view
of ~\cite[\S 2, \S 6]{20} and ~\cite{9}, and give further
interesting remarks in automatic continuity theory.

In \S 3 and \S 4, we summarize the previous developments on
Michael's two problems in view of our affirmative solutions. For
example, we shall discuss the Carpenter's result ~\cite[Corollary
6]{8}. It is a known fact that another equivalent form of
Michael's problem is the long-standing problem of continuity of a
homomorphism from a Fr\'{e}chet algebra $B$ into a semi-simple
Fr\'{e}chet algebra $A$ ~\cite[p. 143]{11}. As we shall see, an
affirmative answer to this problem would give us the shortest
proof of the uniqueness of the Fr\'{e}chet topology for
commutative, semi-simple Fr\'{e}chet algebras. Further, we {\it
completely} settle Forster's problem (1966) in the affirmative
(see ~\cite[REMARK]{22} and Corollary 4.5 below). We also extend
certain automatic (dis)continuity results of ~\cite[\S 12]{15},
{\it completing} the circle of ideas.
\section{Topological tensor (symmetric) algebras over a Banach space} In the
next sections, our important tool is a topological version of the
tensor (symmetric) algebra over a Banach space appearing in
~\cite{26}. However, a fuller description of what we require is
given in ~\cite{14, 12}. We briefly recall some notation, which
will be used throughout the paper. Let $E = \hat{\bigotimes}^1E$
be a Banach space, and take $\hat{\bigotimes}^0E = \cc$. For each
$p \geq 2$, write $\hat{\bigotimes}^pE$ (resp.,
$\check{\bigotimes}^pE$) for the completion of $\bigotimes^pE$
with respect to (briefly: w.r.t.) the projective tensor product
norm $\|\cdot\|_\pi$ (resp., the equicontinuous tensor product
norm $\|\cdot\|_\epsilon$). When it is unnecessary to distinguish,
we may write $\hat{\bigotimes}^pE$ for either completion, and
$\|\cdot\|$ for the specified norm. Then $\widehat{\bigotimes}E$
is a non-commutative, unital Fr\'{e}chet tensor algebra over $E$
(if $\textrm{dim}\,E > 1$) w.r.t. the product
\begin{equation} \label{tensor}
(\sum_pu_p)\,\otimes\,(\sum_pv_p)\,=\,\sum_p(\sum_{i+j=p}u_i\,\otimes\,v_j)
\end{equation}
and the coefficientwise convergence topology defined by an
increasing sequence $(\|\cdot\|_m)$ of seminorms, where
$\|\sum_p\,u_p\|_m\,=\,\sum_{p=0}^m\|u_p\|$. We refer to the
subspaces $\hat{\bigotimes}^pE$ as the homogeneous subspaces of
$\widehat{\bigotimes}E$.

We shall also need in a future proof a commutative analogue
$\widehat{\bigvee}E$ of $\widehat{\bigotimes}E$, and a closed,
linear subspace $\hat{\bigvee}^pE$ of $\hat{\bigotimes}^pE$. Then
$\widehat{\bigvee}E$ is a commutative, unital Fr\'{e}chet
symmetric algebra over $E$ w.r.t. the product
\begin{equation} \label{symmetric}(\sum_pu_p)\,\vee\,(\sum_pv_p)\,=\,\sum_p(\sum_{i+j=p}u_i\,\vee\,v_j)
\end{equation}
and the same coefficientwise convergence topology defined by
$(\|\cdot\|_m)$. We remark that if $\tilde{\sigma} \,=\,\bigoplus
\tilde{\sigma}_{p}$ is the continuous, symmetrizing epimorphism
$\tilde{\sigma} : \widehat{\bigotimes}E \rightarrow
\widehat{\bigvee}E$, then $\textrm{ker}\,\tilde{\sigma}$ is the
closed, two-sided ideal of $\widehat{\bigotimes}E$ generated by
$$\{u\otimes v - v\otimes u : u,\, v \in \hat{\bigotimes}^1E =
E\},$$ so that the algebras in ~\cite{26} and the algebras
$\widehat{\bigvee}E$ are topologically isomorphic ~\cite[p.
315]{14}.

There are some important Banach subalgebras of the above examples
~\cite[Example 2.2.46 (ii)]{12}. Let $\omega$ be a weight on
$\zz^+$ and let $A$ be one of the algebras $\widehat{\bigotimes}E$
and $\widehat{\bigvee}E$. (When it can cause no confusion, the
symbol $(\zz^+)^{<\omega}$ is used to denote the subsemigroup
consisting of all $\zz^+-$valued sequences that are eventually
$0$.) Define $\widehat{\bigotimes}_{\omega}E$ and
$\widehat{\bigvee}_{\omega}E$, respectively, as
\begin{equation} \label{Banach}\{u\,=\,\sum_{p=0}^\infty\,u_p\,\in\,A\,:\,\|u\|_\omega\,:=\,\sum_{p=0}^\infty\|u_p\|\omega(p)\,<\,\infty\}.
\end{equation}
We obtain two unital Banach subalgebras, with
$\widehat{\bigvee}_{\omega}E$ being commutative. These algebras
are called weighted Banach tensor algebras and weighted Banach
symmetric algebras, respectively. The algebra
$\widehat{\bigotimes}_{\omega}E$, where $\omega \equiv 1$ on
$\zz^+$, is the Banach tensor algebra which is the starting point
for the constructions in ~\cite{26}. We recall the following
theorem from ~\cite{14} (also, cf. ~\cite[Satz 2 and Satz 3]{26}),
describing the maximal ideal space and semi-simplicity of
$\widehat{\bigvee}_{\omega}E$.
\begin{thm} \label{Thm 1.2_DM} \begin{enumerate}
\item[{\rm (i)}] The space of characters on
$\widehat{\bigvee}_{\omega}E$ is homeomorphic with the ball
$\{\lambda\,\in\,E{'}\,:\,\|\lambda\|\,\leq\,\inf_{p}\omega(p)^\frac{1}{p}\,\equiv\,\omega_\infty\}\;\;(w^*-\textrm{topology}).$
\item[{\rm (ii)}] $\widehat{\bigvee}_{\omega}E$ (w.r.t. the
equicontinuous tensor norm $\|\cdot\|_\epsilon$) is semi-simple if
and only if $\omega_\infty\,>\,0$. \item[{\rm (iii)}]
$\widehat{\bigvee}_{\omega}E$ (w.r.t. the projective tensor norm
$\|\cdot\|_\pi$) is semi-simple if and only if
$\omega_\infty\,>\,0$ and $E$ has the approximation property.
\end{enumerate}
\end{thm}
Along the same lines, for an increasing sequence
$W\,=\,(\omega_m)$ of weights on $\zz^+$, we define
$\widehat{\bigotimes}_{W}E$ and $\widehat{\bigvee}_{W}E$,
respectively, as
\begin{equation} \label{Frechet} \{u\,=\,\sum_{p=0}^\infty\,u_p\,\in\,A\,:\,p_m(u)\,:=\,\sum_{p=0}^\infty\|u_p\|\omega_m(p)\,<\,\infty\;\textrm{for
all}\, m\,\in\,\nn\}. \end{equation} We obtain two unital
Fr\'{e}chet subalgebras, with $\widehat{\bigvee}_{W}E$ being
commutative. These algebras are called weighted Fr\'{e}chet tensor
algebras and weighted Fr\'{e}chet symmetric algebras,
respectively. Clearly, from \eqref{Frechet}, their Arens-Michael
representations are given by $\widehat{\bigvee}_{W}E = \bigcap_{m
\in \nn}(\widehat{\bigvee}_{\omega_m}E,\, p_m)$ and
$\widehat{\bigotimes}_{W}E$
$=\,\bigcap_{m\,\in\,\nn}(\widehat{\bigotimes}_{\omega_m}E,\,
p_m)$, where $(p_m)$ is a sequence of norms on the respective
Fr\'{e}chet algebra.

Set $E=\ell^1(\zz^+)$, the Banach space for the remaining part of
this section. Then, as in ~\cite[\S 10]{15}, for each $p \in \nn$,
$\hat{\bigotimes}^pE$ can be identified with $\ell^1((\zz^+)^p)$
as a Banach space. This Banach space can also be viewed as the
space of absolutely summable functions on $(\zz^+)^p$
~\cite[Chapter 45]{41} (equivalently, the space $A^+(D^p)$ of
functions in $A(D^p)$ ($D^p$ the product of $p$-copies of $D$ in
$\cc$), with absolutely convergent Taylor series on $\Gamma^p$
(the product of $p$-copies of the unit circle $\Gamma$ in $\cc$),
w.r.t. the $\|\cdot\|_{\ell^1}$-norm on the space).
\begin{thm} \label{p 315_DM2} For each $p \in \nn$, the Banach space $\hat{\bigotimes}^pE$ is a Banach algebra, and the character
space of $\hat{\bigotimes}^pE$ is homeomorphic with the polydisc
$D^p$.
\end{thm}
\noindent{\it Proof.} By ~\cite[Proposition 1.3.11]{12},
$\bigotimes^pE$ is a commutative, unital algebra w.r.t. the
product
\begin{equation}
\label{product-D}(f_1 \otimes \dots \otimes f_p)\cdot(g_1
 \otimes \dots \otimes g_p) = (f_1g_1 \otimes
\dots \otimes f_pg_p)\;(f_i, g_i \in E,\;i \in \nn_p).
\end{equation}
Then, it is easy to see that it is a normed algebra w.r.t. the
projective tensor norm $\|\cdot\|_\pi$ (note that $\|\cdot\|_\pi$
is submultiplicative on $\bigotimes^pE$). Indeed,
$\hat{\bigotimes}^pE$ is a Banach algebra by noticing that the
product can be extended to $\hat{\bigotimes}^pE$ (cf. ~\cite[p.
315]{14}). Thus, for each $p \in \nn$, $\hat{\bigotimes}^pE$ can
be identified with $\ell^1((\zz^+)^p)$ as a Banach algebra via the
homomorphism $$\theta\,:\,(f_1 \otimes \dots \otimes
f_p)(r)\,\mapsto\,f_1(r_1)\cdots f_p(r_p)\; \textrm{where}\;
r\,=\,(r_1,\,\dots,\,r_p)\,\in\,(\zz^+)^p.$$ So, the character
space of $\hat{\bigotimes}^pE$ is homeomorphic with the polydisc
$D^p$ for each $p \in \nn$. $\hfill \Box$

The main point of this identification should be emphasized. The
completion of the tensor power of finite copies of the Banach
space (which is also a singly generated Banach algebra) w.r.t.
$\|\cdot\|_\pi$ is identified with a finitely generated Banach
algebra. Unfortunately, we cannot extend this to $\nn$-fold tensor
product within the framework of Banach algebras. However, we can,
indeed, extend this to $\nn$-fold tensor product within the
framework of Fr\'{e}chet algebras (e.g., the algebra
$\widehat{\bigotimes}E$ and its commutative analogue
$\widehat{\bigvee}E$ below).

As above, $\hat{\bigvee}^pE$ is a closed, linear subspace of
$\hat{\bigotimes}^pE$, consisting of the symmetric elements.
Recall that $\hat{\bigvee}^pE$ is the range of the symmetrizing
map $\tilde{\sigma}_{p}$ (equivalently, the projection of norm
$1$) on $\hat{\bigotimes}^pE$. In fact, every element of
$\hat{\bigotimes}^pE$ is symmetric due to the fact that
$\ell^1((\zz^+)^p)$ is identified with $A^+(D^p)$ and so,
$\partial_kf(z)\;(z\,\in\,D^p)$ will be invariant under
permutations of $1,\,2,\,\dots,\,k$ (cf. ~\cite[p. 320]{14}). So,
$\hat{\bigvee}^pE\,=\,\hat{\bigotimes}^pE$ by ~\cite[Chapter
45]{41} and ~\cite[p. 321]{14}. Thus, as above, for each $p \in
\nn$, $\hat{\bigvee}^pE\,\cong\,\ell^1((\zz^+)^p)$ as a Banach
algebra.

As above, the algebra $\widehat{\bigotimes}E$ is naturally
identified with a graded subalgebra of $\cal B$, and
$\widehat{\bigvee}E$ is naturally identified with a graded
subalgebra of $({\cal B}_{\textrm{sym}}, \vee)$ ~\cite[\S 10]{15},
called the unital Fr\'{e}chet symmetric algebra over $E$. The
latter algebra can also be viewed as the algebra of $p$-absolutely
summable symmetric functions on $(\zz^+)^{<\omega} =
\bigcup_{p\in\nn}(\zz^+)^p$ for all $p \in \nn$. That is, if $f =
\sum_pf_p$ is a symmetric function on $(\zz^+)^{<\omega}$, then
$f_p$ is an absolutely summable function on $(\zz^+)^p$ for each
$p \in \nn$. The former algebra $\widehat{\bigotimes}E$ can also
be viewed as the algebra of $p$-absolutely summable functions on
$(\zz^+)^{<\omega}$.\\

\noindent {\it Remarks A.} 1. We shall often write the Banach
algebra $\widehat{\bigvee}_{\omega}E$, where $\omega \equiv 1$ on
$\zz^+$, as ${\cal U}_{1} = \ell^1((\zz^+)^{<\omega})$ throughout
the paper (see \S 3 below), since it is isometrically isomorphic
with $\ell^1((\zz^+)^{<\omega})$ ~\cite[p. 142]{15}. So, by
Theorem 2.1, it may be viewed as the algebra of functions in
$A(\ell_{[1]}^\infty)$ with absolutely convergent Taylor series
(equivalently, absolutely summable symmetric functions on
$(\zz^+)^{<\omega}$), where $\ell_{[1]}^\infty$ is the unit
polydisc in $\cc^\infty$; compare ~\cite{9} and the above
situation in the finite-dimensional case.

\noindent 2. Recall that Dales, author and Read gave a very long
proof of showing $\ell^1((\zz^+)^{<\omega})$ a Banach algebra of
power series (briefly: BAPS) ~\cite[Theorem 10.1 (i)]{15}. We
drastically reduce this proof in the following
\begin{lem} \label{Thm 10.1 (i)_DPR} There is a topological isomorphism
$\theta$ of $\ell^1((\zz^+)^{<\omega})$ onto $\ell^1(\zz^+)$, and
so the Banach algebra $\ell^1((\zz^+)^{<\omega})$ is
(isometrically isomorphic to) a semi-simple Banach algebra of
power series $\ell^1(\zz^+)$.
\end{lem}
\noindent {\it Proof.} First, we note that the homomorphism,
discussed in the proof of ~\cite[Theorem 1.2 (iii)]{14}, from
$\widehat{\bigvee}_{\omega}E$, where $\omega \equiv 1$ on $\zz^+$,
into $(\F, \,\omega)$ (which is, in fact, $\ell^1(\zz^+)$), is,
indeed, an injective epimorphism. Moreover, by ~\cite[\S 9]{24},
it is continuous since $\ell^1(\zz^+)$ is a BAPS in $\F$. So, by
the open mapping theorem, this isomorphism is a topological
isomorphism between ${\cal U}_{1} = \ell^1((\zz^+)^{<\omega})$ and
$\ell^1(\zz^+)$. Thus the norm $q_1$ of ${\cal U}_{1}$ is
equivalent to $\|\cdot\|_{\ell^1}$ on a semi-simple BAPS
$\ell^1(\zz^+)$ (we shall use this fact in the proof of the main
theorem below). $\hfill \Box$

However, the difference between the above proof and the proof of
~\cite[Theorem 10.1 (i)]{15} is as follows. In the above proof,
the bounded linear functional $\lambda_k$ on
$\textrm{SE}_\pi(k),\;k\,\in\,\nn$, need not be higher point
derivations of infinite order on $\ell^1(\zz^+)$ (cf. ~\cite[\S 1,
p. 169]{13}, introduction in ~\cite[\S 2]{14} and ~\cite[Theorem
10.1]{15}). We shall return to a reduction of the proof of
~\cite[Theorem 10.1 (ii)]{15} (the Fr\'{e}chet case) in Corollary
4.9 below.

\noindent 3. Richard Aron and his team have worked a lot on the
algebras of analytic functions on an infinite-dimensional Banach
space ~\cite{6, 7}. We, here, represent a ``tensor approach" for
the study of such algebras.
\section{Our approach}
\noindent {\bf 3.1. Background.} Our approach is based on
generating a totally discontinuous higher point derivation
$(d_n)_{n \in \zz^+}$ of infinite order on the test algebra $\cal
U$ at a discontinuous character $\phi$. We then apply the
Dales-author-Read's method from ~\cite{15} to arrive at a
contradiction. For general information about higher point
derivation $(d_n)$ of infinite order on a topological algebra $A$
at a character, we refer to ~\cite{13, 11}. We call $(d_n)$
``natural" on a topological algebra $A$ at a character $d_0$
(continuous or not), if $d_1$ is induced by $d_0$. In other words,
$d_1$ is continuous (resp., discontinuous), if so is $d_0$. Thus
this notion extends the notion of totally discontinuous $(d_n)$,
stated in ~\cite{14, 15}. In particular, for us, if $(d_n)$ is
natural and totally discontinuous, then all $d_n (n \in \zz^+)$
are discontinuous. Of course, if $(d_n)$ is natural and
continuous, then all $d_n (n \in \zz^+)$ are continuous ~\cite[p.
129]{15}. For example of natural $(d_n)$, we refer to
~\cite[Example 6.4]{11}. In the Fr\'{e}chet case, one can take
$d_n(f) = \frac{f^{(n)}(z)}{n!}$, where $f \in \textrm{Hol}(U),\;
U$ the open unit disc in $\cc$, and $d_n(f) =
\frac{f^{(n)}(x)}{n!}$, where $f \in C^\infty([0, 1])$ ~\cite[pp.
177, 188]{13}. For examples of non-natural $(d_n)$, we refer to
~\cite[\S 2]{14}. In the Fr\'{e}chet case, one can take $d_n =
\lambda_n \circ P_n$ on the weighted Fr\'{e}chet symmetric algebra
$\widehat{\bigvee}_{W}E$ (discussed in \S 2; also, called
Dales-McClure-Fr\'{e}chet algebra in ~\cite{34}) at a continuous
character $d_0 = P_0$, where $\lambda_1$ is {\it any}
discontinuous linear functional on $E = \ell^1$ (e.g., $\lambda$
in the proof of ~\cite[Theorem 12.3]{15}).

It is well-known that every higher point derivation $(d_n)$ of
infinite order on the disc algebra $A(D)$ (resp., $H^\infty(U)$)
at a character is always natural and continuous. The same holds
for $\textrm{Hol}(U)$ and $\F$ in the Fr\'{e}chet case ~\cite[\S
8]{11} (also, see ~\cite[Theorem 3.7]{13} for the Fr\'{e}chet
algebra $C^\infty([0, 1])$). Moreover, one can claim the same
results for the analogous examples in the finitely many variables
case ~\cite[\S 8]{11}, but certainly not in the infinitely many
variables case, as discussed in the above paragraph. Thus, it is
crucial for our approach to consider only ``natural" higher point
derivation $(d_n)$ of infinite order on $\cal U$ at a character
$d_0$ (continuous or not). For such a natural $(d_n)$ on $\cal U$,
there is a (continuous or not) linear functional $\lambda_1$ on
$E$ such that $d_1 = \lambda_1 \circ P_1$, since $\cal U$ is a
weighted Fr\'{e}chet symmetric algebra over $E$ as below and $P_1$
is a projection onto $E$ (see introduction in \S 2 before
~\cite[Theorem 2.1]{14}). Further, we remark that if $(d_n)$ is
discontinuous on $\cal U$ (i.e., if at least one of the $d_n$ is
discontinuous), then, by the identities (2.1) of ~\cite{14},
$(d_n)$ must be totally discontinuous higher point derivation of
infinite order. As a consequence, if $(d_n)$ is natural and
discontinuous on $\cal U$, then $(d_n)$ is natural and totally
discontinuous on $\cal U$ in our sense; equivalently, the
character $d_0$ must be discontinuous. The converse is trivial,
since in that case, $d_1$ is discontinuous. So, if a character is
continuous on $\cal U$, then either $(d_n)$ is not natural (this
situation arises in the Banach case as well ~\cite[\S 2]{14}), or
$(d_n)$ is natural and continuous (i.e., $\cal U$ is a FrAPS
~\cite[Theorem 10.1 (ii)]{15}).\\

\noindent {\bf 3.2. Proofs of the main theorem and crucial
remarks.} We now discuss our approach, in order to solve Michael's
problem affirmatively. We first see that the test case $\cal U$ is
a Fr\'{e}chet algebra of the kind $\widehat{\bigvee}_{W}E$, where
$E = \ell^1(\zz^+)$ and $W = (\omega_m)$ is an increasing sequence
of weights on $\zz^+$ (note that the germ of the ideas for our
approach lies in ~\cite{34}). Following ~\cite[Definition 9.2 and
\S 10]{15}, we have Banach algebras $\ell^1(S,\, \omega) = {\cal
U}_1$, where $\omega \equiv 1$ on $S = (\zz^+)^{<\omega}$, and for
each $m \geq 2$, $\ell^1(S,\,\omega_m) = {\cal U}_{m} \cong
\widehat{\bigvee}_{\omega_{m}}E$, where $\omega_m$ is a weight on
$\zz^+$ defined by $\omega_m(|r|) = m^{|r|},\;r \in S =
(\zz^+)^{<\omega}$ (or, one may take $\omega_m$ as a weight on $S$
defined by $\omega_m(r) = m^{|r|}$). Then, $${\cal U} =
\bigcap_{m\in\nn}{\cal U}_m = \ell^1(S,\, W) \cong
\bigcap_{m\in\nn}\widehat{\bigvee}_{\omega_{m}}E =
\widehat{\bigvee}_{W}E.$$ All these algebras are graded
subalgebras of $\F_\infty$. We remark that the map $\epsilon$,
restricted to ${\cal U}_m$, is an isometric isomorphism of ${\cal
U}_m$ onto $\widehat{\bigvee}_{\omega_{m}}E$ (for $m\,=\,1$, see
Remarks A. 1. above) and the same map $\epsilon$, restricted to
$\cal U$, is an isometric isomorphism of ${\cal U}$ onto
$\widehat{\bigvee}_{W}E$ ~\cite[p. 142]{15}.

It is important to note the following chains of (dense) continuous
inclusions of certain algebras and their corresponding character
spaces. We have
\begin{equation} \label{inclusion}
{\cal U}\,\hookrightarrow\,{\cal U}_{m}\,\hookrightarrow\,{\cal
U}_{1}\,=\,\ell^1(S,
\{1\})\,\cong\,\widehat{\bigvee}_{\{1\}}E\,\hookrightarrow\,\widehat{\bigvee}E,
\end{equation}
where the density follows from the fact that all algebras are
countably generated by the monomials $X_1, X_2,\dots$, and the
Arens-Michael representation of $\cal U$. By ~\cite[Lemma
3.2.5]{23}, we have
\begin{equation} \label{spectra}
M(\widehat{\bigvee}_{\{1\}}E)\,\cong\,\ell_{[1]}^\infty\,\cong\,M(\ell^1(S,\,
\{1\}))\,\hookrightarrow\,M({\cal
U}_m)\,\cong\,\ell_{[m]}^\infty\,\hookrightarrow\,M({\cal
U})\,\cong\,\ell^\infty
\end{equation} ($w^*$-topology), where the homeomorphisms are due to Theorem 2.1
and ~\cite{17}, respectively.

Next, we recall that the test case $\cal U$ is, indeed, a FrAPS in
$\F$. We {\it claim} that there exists a natural, discontinuous
higher point derivation $(d_n)$ on $\cal U$, induced by a
discontinuous character $d_0 \,=\,\phi$. Such a $(d_n)$ obviously
turns out to be non-degenerate and totally discontinuous in our
sense, as discussed above. Note that if we prove our claim we have
our desire result. For since such a $(d_n)$ gives an epimorphism
$\theta$ from $\cal U$ onto $\F$ by ~\cite[Theorem 11.2]{15}, we
then apply the Dales-author-Read's method, used in the proof of
~\cite[Theorem 10.1]{15}, in order to show that $\theta$ is,
indeed, an isomorphism (i.e., $\cal U\,=\,\F$). This contradicts
the fact that $\cal U$ is a semi-simple algebra and $\F$ is a
local algebra or contradicts ~\cite[Theorem 10.1 (ii)]{15} and
~\cite[Corollary 3.4]{31} (cf., also, Remarks D. 5 below). The
crucial point of the main theorem should be emphasized. It is a
surprising consequence of the fact that one is able to show that
the test case $\cal U$ is, indeed, a weighted Fr\'{e}chet
symmetric algebra $\widehat{\bigvee}_{W}E$ as above. So if one
starts with a discontinuous character on $\cal U$, then one can
construct $(d_n)$ by applying the Dales-McClure method ~\cite{14,
12} in the following

\noindent {\it Proof of the main theorem.} Assume that there is a
discontinuous character $\phi$ on ${\cal
U}\,=\,\bigcap_{m\,\in\,\nn}{\cal U}_m$. Then, for each
$m\,\in\,\nn$, $\phi$ is a discontinuous character on the normed
algebra $({\cal U}, q_m)$, because if it is continuous for some
$m$, then there is $c\,>\,0$ such that
$$|\phi(f)|\,\leq\,cq_m(f), \;f\,\in\,({\cal U},\,(q_m)),$$ a
contradiction of the assumption that $\phi$ is discontinuous on
$\cal U$. Now, $\phi$ can be extended to a discontinuous linear
functional $\lambda_1$ on $({\cal U}_1, q_1)$ using a Hamel basis
of a complement of $\cal U$ in ${\cal U}_1$ (remark that there are
infinitely many ways of extending $\phi$; also, we cannot extend
$\phi$ as a discontinuous character on ${\cal U}_1$). By Lemma
2.3, the norm $q_1$ is equivalent to $\|\cdot\|_{\ell^1}$ on
$\ell^1$, and so, w.l.o.g., $\lambda_1$ is also a discontinuous
linear functional on $\ell^1\,=\,\hat{\bigvee}_{\{1\}}^{1}E$.
Alternatively, by a remark below ~\cite[Corollary 10]{9}, the
existence of a discontinuous character on $\cal U$ induces a
discontinuous linear functional on $\ell^1$.

Now, by ~\cite[Theorem 2.1 (ii)]{14}, there are linear functionals
$\lambda_n$ on $\textrm{SE}_\pi(n)$, for $n = 2, 3, \dots$, such
that (2.1) of ~\cite{14} holds for any positive integers $n$ and
$m$, $u \in \textrm{SE}_\pi(n)$ and $v \in \textrm{SE}_\pi(m)$. As
discussed in the beginning of ~\cite[\S 2]{14}, define $d_n
=\lambda_n \circ P_n$, for $n \in \nn$. Then $(d_n)$ is a
``natural", discontinuous higher point derivation of infinite
order on $\cal U$ at a discontinuous character $d_0 = \phi$ (note
that since $\cal U$ is a subalgebra of $\widehat{\bigvee}E$, we
are free to apply arguments from ~\cite[\S 2]{14}). Moreover,
$(d_n)$ is, indeed, non-degenerate and totally discontinuous in
our sense. This is clear since $d_1$ is discontinuous on $\cal U$
as $d_0$ is discontinuous on $\cal U$ (see the proof of
~\cite[Theorem 11.2]{15}), and the discontinuity of the remaining
$d_n$ on $\cal U$ also follows from the fact that each $\lambda_n$
($n \geq 2$) is discontinuous since $\lambda_1$ is discontinuous
and the fact that (2.1) of ~\cite{14} holds.

Thus we have the required existence of a natural, discontinuous
higher point derivation $(d_n)$ of infinite order on $\cal U$ at a
discontinuous character $\phi$. By ~\cite[Theorem 11.2]{15}, this
$(d_n)$ induces an epimorphism $\theta$ from $\cal U$ onto $\F$.
Now, we show that $\theta$ is, indeed, an injection. Our argument
here is kept short because it uses the key ideas from the proof of
~\cite[Theorem 10.1]{15}, but the map $\theta$ that we are working
with, is entirely different (i.e., induced by a natural, totally
discontinuous $(d_n)$ in our sense).

Our first remark is the following. There is a sequence
$(g_i)_{i\,\in\,\nn}$ of singular elements in $\F$ such that
$\theta(X_i)\,=\,g_i$. We claim that at least one member of this
sequence has the order $1$. Indeed, assume towards a contradiction
that $o(g_i)\,>\,1$ for all $i\,\in\,\nn$. Since
$(X_i)_{i\,\in\,\nn}$ generates the Fr\'{e}chet algebra $\cal U$,
for any singular element $f$, $o(\theta(f))\,>\,1$; equivalently,
$X$ is not in the range of $\theta$, a contradiction of the fact
that $\theta$ is an epimorphism. In fact, w.l.o.g., we may assume
that $\theta(X_1)\,=\,g_1\,=\,X$, since if $g_1$ has the initial
form $X$ in $\F$, then there is an automorphism $\alpha$ of $\F$
such that $\alpha(g_1)\,=\,X$, and in this case,
$\alpha\,\circ\,\theta\,=\,\Theta$ is a discontinuous epimorphism
such that $\Theta(X_1)\,=\,X$.

Our main {\bf claim} is that we can choose the sequence
$(g_i)_{i\,\in\,\nn}$ of singular elements in $\F$ with
$g_1\,=\,X$ such that $o(g_i)\,\geq\,i\;(i\,\in\,\nn)$ so that the
map $\theta$ is an injection. Incidentally, it is a somewhat
surprising fact that we can follow the proof of ~\cite[Theorem
10.1]{15} in our case since the arguments there are purely
set-theoretic and algebraic while dealing with selection of
$(g_i)$ in $\F$ with the stated conditions and ~\cite[Lemma
10.2]{15}. We give a proof for the sake of completeness.

The first step in our construction is to specify a function
$\gamma\,:\,\nn\,\rightarrow\,\nn$ with the following properties:
we have $\gamma_i\,\leq\,i\; (i\,\in\,\nn)$ and, for each
$n\,\in\,\nn$ and each $r\,=\,(r_1,\dots,\,r_n)\,\in\,\nn^n$,
there exists $k\,\in\,\nn$ such that
\begin{equation} \label{10.2} (\gamma(k + 1),\dots,\,\gamma(k +
n))\,=\,(r_1,\dots, \,r_n).
\end{equation}
Such a function is easily constructed by listing all the elements
in the countable set
$$\bigcup\{(r_1,\dots,\,r_n)\,\in\,\nn^n\,:\,n\,\in\,\nn\}$$ in
one sequence and by regarding the elements in this listing as
successive parts of a function in $\nn^\nn$. Note that, in this
case, there are infinitely many values of $k\,\in\,\nn$ such that
equation (\ref{10.2}) holds for each specified value of $r$.

For each $i\,\in\,\nn$ with $i\,\geq\,2$, we define
$$E_i\,=\,\{j\,\in\,\nn\,\setminus\,\{1\}\,:\,\gamma(j)\,=\,i\},$$ and we
take $E_1\,=\,\{1\}$ so that $\{E_i\,:\,i\,\in\,\nn\}$ is a
partition of $\nn$, and each $E_i$ save for $E_1$ is infinite.
Further, $\textrm{min}\,E_i\,\geq\,i\;(i\,\in\,\nn)$.

We now take a 'rapidly increasing sequence'
$(c_i\,:\,i\,\in\,\nn)\,\in\,\nn^\nn$ with $c_1\,=\,1$. In fact,
we shall write $(c_j\,:\,j\,\in\,\nn)$ as
$(a_1,\,b_1,\,a_2,\,b_2,\dots)$, where
$$1\,=\,a_1\,<\,b_1\,<\,a_2\,<\,b_2\,<\,\cdots.$$ The growth
conditions that we shall impose are:
\begin{equation} \label{10.3} a_{i + 1}\,>\,ia_i\;\;\;\;(i\,\in\,\nn)
\end{equation}
and
\begin{equation} \label{10.4} b_i\,>\,i\cdot(i(1 +
a_i))!\,\cdot\,i^{i(1 + a_i)}\,\cdot\,b_{i-1}^{i(1 +
a_i)}\;\;\;\;\,(i\,\geq\,2).
\end{equation}
Clearly, we can choose the sequence $(c_j\,:\,j\,\in\,\nn)$ to
satisfy these constraints.

For each $i\,\in\,\nn$, we define
\begin{equation} \label{10.5}
g_i\,=\,\sum\{b_jX^{a_j}\,:\,j\,\in\,E_i\}\,\in\,M\,\subset\,\F,
\end{equation}
where $M$ is the unique maximal ideal. Note that, since $a_i \geq
i$ and $\textrm{min}\,E_i \geq i$ for each $i \in \nn$, we have
$o(g_i) \geq i\;(i \in \nn)$, as required in the above remarks.

Our claim will follow easily from the following lemma ~\cite[Lemma
10.2]{15}, whose proof we omit. We continue to denote the
semigroup $(\zz^+)^{<\omega}$ by $S$.
\begin{lem} \label{lem_10.2} Let $m \in \nn$. Let
$(r_1,\dots, r_m, 0, 0,\dots) \in S$ be such that $r_1 \leq r_2
\leq \cdots \leq r_m,$ and let $k \in \nn$ be such that $k
> m$ and $(\gamma(k + 1),\dots, \gamma(k +
m)) = (r_1,\dots,\,r_m)$. Set $P\,=\,\sum_{i = 1}^m a_{k + i}$ and
$Q\,=\,\prod_{i = 1}^m b_{k + i}$.
\begin{enumerate}
\item[{\rm (i)}] We have $$\pi_P(g_{r_1}\cdots
g_{r_m})\,\geq\,Q.$$ \item[{\rm (ii)}] Provided that the sequence
$(c_j\,:\,j\,\in\,\nn)$ satisfies equations (\ref{10.3}) and
(\ref{10.4}), we have $$\pi_P(g_{s_1}\cdots g_{s_n})\,\leq\,Q/k,$$
for each $(s_1,\dots, s_n, 0, 0,\dots) \in S$ with $\{s_1,\dots,
s_n\} \neq \{r_1,\dots, r_m\}.$
\end{enumerate}
\end{lem}
We now prove that the fact that $\theta$ is injective follows from
Lemma \ref{lem_10.2}.

Each element $f\,\in\,\cal U$ can be written in the form
$$f\,=\,\sum \beta_{(t_1,\dots,\,t_m)}\, X_{t_1}\cdots \,X_{t_m},$$ where
$t_1 \leq t_2 \leq \cdots \leq t_m$. Take such an element with $f
\neq 0$; we shall show that $\theta(f) \neq 0$. We may suppose for
convenience that $q_1(f) = 1$ since $f \in {\cal U} \subset {\cal
U}_1$ (recall that $q_1 = \|\cdot\|_{\ell^1(S)}$ and ${\cal U}_1 =
\ell^1(S)$ in our notation). Choose a specific element $t =
(t_1,\dots, t_m, 0, 0, \dots) \in S$ for which $\beta_t \neq 0$.
Then there exists $k\, \in\, \nn$ with $k\,
>\, 1/|\beta_t|$ and such that $(\gamma(k + 1),\dots, \gamma(k +
m))\, =\, (t_1, \dots, t_m)$. Define $P$ and $Q$ with respect to
the elements $t \in S$ and $k \in \nn$ as in Lemma \ref{lem_10.2}.
By clauses (i) and (ii) of that lemma, we have
\begin{equation} \label{inequality 1} |\pi_P(\beta_tg_{t_1}\cdots\,g_{t_m})|\,\geq\,Q|\beta_t|,
\end{equation}
and
\begin{equation} \label{inequality 2} \begin{split}
|\pi_P(\theta(f) - \beta_tg_{t_1}\cdots g_{t_m})| & =
|\pi_P(\sum_{s \in S}\beta_sg_{s_1}\cdots g_{s_n} : \{s_1,\dots,
s_n\} \neq
\{t_1,\dots, t_m\})|\\
& \leq \textrm{sup}\{\pi_P(g_{s_1}\cdots\,g_{s_n} : \{s_1,\dots,
s_n\} \neq \{t_1, \dots, t_m\}\}\\
& \leq Q/k,
\end{split}
\end{equation}
where we recall that $\sum_{s \in S}|\beta_s|\,=\,1$. It follows
that
$$|\pi_P(\theta(f))|\,\geq\,Q\,\cdot\,(|\beta_t| -
1/k)\,>\,0,$$ and so $\theta(f) \neq 0$ in $\F$, as required to
complete the proof of injectivity of $\theta$.

Thus, $\cal U = \F$, a contradiction as discussed in the preceding
paragraph to the start of proof. Hence all characters on $({\cal
U}, (q_m))$ are continuous. By ~\cite[Proposition 7]{9} or
~\cite[Proposition 2.5]{20}, they are, indeed, the point
evaluation mappings, i.e., if $\phi$ is a character on $\cal U$,
then there exists $z \in \ell^\infty$ such that $\phi = \phi_z$,
where $\phi_z(f) = f(z)$ for all $f \in \cal U$. This completes
our proof. $\hfill \Box$\\

\noindent {\it Remarks B.} 1. It is essential to emphasize two
important remarks of our proof as follows. First, we can choose
such a sequence $(g_i)$ in $\F$ since the range of $\theta$ is the
whole $\F$, which allows necessary freedom in selection of
$(g_i)$. Second, having had the sequence $(g_i)$ in $\F$ with the
stated conditions, it is not necessary that the map $\theta$ on
$\cal U$ would always be a unique continuous injection as in
~\cite[Theorem 10.1]{15}; one may be able to construct a unique
discontinuous isomorphism as seen in the proof above. In other
words, the uniqueness of the injective map $\theta$ in the proof
of ~\cite[Theorem 10.1]{15} is up to the continuity, and the
uniqueness of the discontinuous isomorphism $\theta$ in our proof
is up to the injectivity. This means that if we drop the
continuity, then we have a discontinuous epimorphism
~\cite[Theorem 11.2]{15}, and if we drop the injectivity, then we
have a continuous epimorphism ~\cite[Theorem 4.4.9 (ii)]{12} and
~\cite[Theorem 3.7]{13}. Here is another possibility: one might
have a unique continuous isomorphism as seen in ~\cite[Theorem
4.4.9 (ii)]{12} (consider the quotient map by
$\textrm{ker}\,\theta$, discussed in ~\cite[4.3]{32} for this
case). In other words, the uniqueness of the continuous
isomorphism in ~\cite[4.3]{32} is up to the injectivity. This
means that if we drop the injectivity, then we have a
discontinuous epimorphism ~\cite[Theorem 11.2]{15}. In this case,
we also may have a discontinuous epimorphism, induced by a
``non-natural", discontinuous $(d_n)$ on $\cal U$; see 3. below
for details.

\noindent 2. The author asked whether every (surjective)
homomorphism $\theta : B \rightarrow \F$ from a non-Banach
Fr\'{e}chet algebra $B$ is continuous ~\cite[p. 135]{31}. Then,
Dales, author and Read further took up this question, in order to
affirmatively answer the Dales-McClure's problem ~\cite[Theorem
11.2]{15}. Now, we apply the Dales-McClure's method and the
Dales-author-Read's method to prove the main theorem. This clearly
establishes the fact that how all this works have turned out to be
the stepping stones for the complete solution to Michael's problem
since 2004, as discussed in Introduction.

\noindent 3. In view of our approach, some crucial remarks on the
Dales-author-Read's method are in order. As discussed in the
beginning of the section (resp., in ~\cite[\S 2]{14} in the Banach
case), one may start with {\it any} discontinuous linear
functional $\lambda_1$ on $\ell^1$, in order to produce a
``non-natural", totally discontinuous $(d_n)$ of infinite order on
$\cal U$ (resp., on ${\cal U}_m$ for each $m$ in the Banach case)
at a continuous character $d_0 = P_0$. This would then lead us to
a discontinuous epimorphism $\theta$ from $\cal U$ (resp., ${\cal
U}_m$ for each $m$) onto $\F$ by ~\cite[Theorem 11.2]{15} (resp.,
~\cite[Theorem 2.3]{14}). This $\theta$ cannot be injective,
because $\textrm{ker}\,\theta$ is a prime ideal of $\cal U$
(resp., of ${\cal U}_m$ for each $m$), which is dense in
$\textrm{ker}\,d_0$ of $\cal U$ (resp., of ${\cal U}_m$ for each
$m$), such that the quotient algebra ${\cal
U}/\textrm{ker}\,\theta$ (resp., ${\cal U}_m/\textrm{ker}\,\theta$
for each $m$) is isomorphic to $\F$ (see ~\cite[Remark 3-17
(2)]{20} for the Banach case). Thus the Dales-author-Read's method
has been used to show that the map $\theta$ is an injection in two
extreme cases: (1) to show that the test case $\cal U$ for
Michael's problem is, indeed, a FrAPS in $\F$, where the
corresponding $(d_n)$ is a natural, continuous higher point
derivation of infinite order on $\cal U$ at a continuous character
$d_0 = P_0$; and (2) to solve Michael's problem in the
affirmative, where the corresponding $(d_n)$ is a natural, totally
discontinuous higher point derivation of infinite order on $\cal
U$ at a discontinuous character $d_0 = P_0$ so as to arrive at a
contradiction as discussed in the proof above (and below).

\noindent {\it Alternative proof of the main theorem.} We bypass
the proof of showing $\theta$ injective in the above proof as
follows. First, the {\it separating space}, ${\cal S}(T)$, of the
linear map $T : E \rightarrow F$, where $E$ and $F$ are
$(F)-$spaces, is defined as
$${\cal S}(T) = \{y \in F : \textrm{lim}_{n \rightarrow
\infty}Tx_n = y\; \textrm{for some sequence}\; (x_n) \in E\;
\textrm{with} \;x_n \rightarrow 0\}.$$ By ~\cite[Proposition
5.1.2]{12}, ${\cal S}(T)$ is a closed linear subspace of $F$; and
$T$ is continuous if and only if ${\cal S}(T)\,=\,\{0\}$. Now
suppose that $A$ and $B$ are $(F)-$algebras and that
$\theta\,:\,B\,\rightarrow\,A$ is a homomorphism such that
$\theta(B)$ is dense in $A$. Then it is easily checked that ${\cal
S}(\theta)$ is a closed ideal in $A$ ~\cite[\S 6]{15} (see
~\cite[Proposition 5.1.3]{12} for the Banach algebra case).

Following the first proof, we have a discontinuous epimorphism
$\theta : {\cal U} \rightarrow \F$, induced by a natural, totally
discontinuous higher point derivation $(d_n)$ of infinite order on
$\cal U$ at a discontinuous character $d_0 = \phi$, with
$\theta(X_1) = X$. Then the separating space ${\cal S}(\theta)$ is
a closed ideal in $\F$. Hence, by ~\cite[Example 1]{4}, ${\cal
S}(\theta) = \F X^n \subset \F X$ for some $n\,\in\,\nn$, and so
its preimage is ${\cal U}X_1^n$. Further, since $d_0$ is a
discontinuous character on $\cal U$, the separating space ${\cal
S}(d_0)$ is a non-zero closed ideal in $\cc$, and so ${\cal
S}(d_0)\,=\,\cc$, with $\cal U$ as its preimage. Moreover,
$\textrm{ker}\,d_0$ is a dense maximal ideal in $\cal U$. Now, as
discussed above, $\textrm{ker}\,\theta$ is a prime ideal of
infinite codimension in $\cal U$, which is dense in
$\textrm{ker}\,d_0$, such that the quotient algebra ${\cal
U}/\textrm{ker}\,\theta$ is isomorphic to $\F$. Hence
$\textrm{ker}\,\theta$ is also dense in $\cal U$. Since
$\textrm{ker}\,\theta \subset {\cal U}X_1^n \subset {\cal U}X_1$,
and, by ~\cite[p. 546]{20}, ${\cal U}X_1$ is a closed, prime ideal
in $\cal U$, we have a contradiction of the fact that
$\textrm{ker}\,\theta$ is also dense in $\cal U$. This completes
the proof of the main theorem. $\hfill \Box$

We remark that Dales and McClure cooked their method up to show
that the statement (C) of ~\cite{13} is false, leading to an
important contribution in automatic continuity theory (cf. Theorem
2.2 and a question in ~\cite{14}). Thus this method provides an
ingenious construction of linear functionals $\lambda_n$ on
$\textrm{SE}_\pi(n)$ (for $n\,\geq\,2$), and was used on purpose
(so is in our case as well).

\noindent 4. There are no characters $\phi$ on ${\cal U} =
\textrm{Hol}(\ell^\infty)$, not equal to evaluation at $z \in
\ell^\infty$, such that $\phi(f) = f(z)$ whenever $f \in
\textrm{Hol}(\ell^\infty)$ is finitely determined
(cf.~\cite[Theorem 9]{9}). More generally, by~\cite[Theorem
4.4]{10}, there are no characters $\phi$ on $\Upsilon(\ell^\infty)
= O(\ell^\infty)$, not equal to evaluation at $z \in \ell^\infty$,
such that $\phi(\tilde{f}) = f(z)$ whenever $\tilde{f} \in
\Upsilon(\ell^\infty) = O(\ell^\infty)$ is a finitely determined
function germ.

\noindent 5. Clayton showed that the continuous character on $\cal
U$ is the point evaluation mapping and vice versa
~\cite[Proposition 7]{9}. Dales, author and Read showed that the
character $\pi_0$ on a FrAPS $\cal U$ is continuous
~\cite[Corollary 11.5]{15}. These two results jointly were another
starting point for author to affirmatively answer Michael's
problem as follows.

Clearly, for each $z \in \ell^\infty$, $\textrm{ker}\,\phi_z = \{f
\in {\cal U} : f(z) = 0\}$ is a closed maximal ideal of $\cal U$.
On the other hand, by identifying $f$ with its Taylor series
expansion about $z \in \ell^\infty$, $\textrm{ker}\,\pi_0 = \{f
\in {\cal U} : \pi_0(f) = f(z) = 0\}$, the closed maximal ideal of
$\cal U$ consisting of $f$ vanishing at $z \in \ell^\infty$, since
$\pi_0$ is a continuous character. So, $\pi_0 = \phi_z$ for some
$z \in \ell^\infty$. Recall that author asked \.Zelazko in 2004
[Personal Communication] whether every FrAPS is functionally
continuous. The reason to ask this question was that author had a
strong belief that it should be very difficult to find a FrAPS
with a discontinuous character.

As a corollary, we have an affirmative solution to the Michael's
problem.
\begin{cor} \label{Cor Michael} Every commutative Fr\'{e}chet
algebra is functionally continuous.
\end{cor}
\noindent {\it Remarks C.} 1. The author extended the notion of
Stein (resp., Riemann) algebras to locally Stein (resp., locally
Riemann) algebras in ~\cite{36} (resp., in ~\cite{32}). As an
application of this corollary, we see that every character on a
(locally) Stein (resp., Riemann) algebra is necessarily
continuous, extending the results of ~\cite{22, 19}. In
~\cite{19}, Ephraim further extended the Markoe's result by
exposing the elementary nature of the Forster's theorem (Markoe's
result extends the Forster's result by taking the dimension of the
singularity set of $X$ finite in place of the dimension of $X$
finite). Recall that any character on a Stein algebra $A$ is a
point evaluation map if $M(A)$ is finite-dimensional
~\cite[Theorem 4]{22}, and Ephraim extended this result
~\cite[Theorem 2.3]{19}. Not only this, but $X\,=\,M(A)$ in view
of ~\cite[Theorem 1]{22}, that is, $X$ is a domain of holomorphy
(cf.~\cite[Chapters 2, 11, 12]{23}). In a more general case,  by
~\cite[Corollary 4.4]{36}, any character on a semi-simple locally
Stein algebra $A$ is a point evaluation map if $Y$ is dense in
$M(A)$ (this hypothesis is the Corona theorem in the theory of SCV
while solving Gleason's problem; follow discussion after Corollary
4.4).

\noindent 2. We have a few more interesting consequences as
follows. First, recall that Michael also asked whether every
character on a commutative, complete LMC-algebra is bounded
~\cite[\S 12, Question 2]{29}. Dixon and Fremlin showed that the
two problems are, in fact, equivalent ~\cite{18}. Akkar gave in a
nice interpretation of this fact: the two problems are equivalent
because every complete LMC algebra is, as a ``bornological
algebra", isomorphic to inductive limit of a family of Fr\'{e}chet
algebras ~\cite{1}. The equivalence of the two questions also
follows from the fact that there is a complete LMC-algebra,
produced by Craw ~\cite{10} (see below), such that the existence
of a discontinuous character on some commutative, unital
Fr\'{e}chet algebra would imply the existence of an unbounded
character on the complete LMC-algebra. We have
\begin{cor} \label{Michael_LMC} Every commutative,
complete LMC-algebra is functionally bounded.
\end{cor}
We remark that it is easy to find a non-metrizable, commutative,
complete LMC-algebra that is not functionally continuous
~\cite[Proposition 12.2 and Remark]{29}. That is, there is a
bounded, discontinuous character on this algebra; however, every
continuous character on $C(T)$ is a point evaluation mapping for
some $t \in T$ ~\cite[\S 4]{10}. Surprisingly, below, we give an
example of a commutative, non-metrizable, non-complete LMC-algebra
that is functionally continuous and functionally bounded. We take
this opportunity to correct an obvious typo in ~\cite[Theorem 9.3
(b)]{15} as follows (cf. ~\cite{9} and ~\cite[Theorem 2.7]{20}).
\begin{cor} \label{spectrum_test case} Let
$\lambda\,=\,(\lambda_n)\,\in\,\ell^{1}\,\setminus\,c_{00}$, and
let $g\,=\,\sum_{n=1}^\infty\lambda_nX_n$. The following
statements are equivalent. \begin{enumerate} \item[{\rm (i)}]
There are no non-zero characters on the quotient algebra ${\cal
M}/{\cal I}$, where $\cal M$ is the closed maximal ideal
$\{f\,\in\,{\cal U}\,:\,f(0, 0, \dots)\,=\,0\}$ and ${\cal
I}\,=\,\bigcup_{n\,\in\,\nn}{\cal I}_n$, where ${\cal
I}_n\,=\,X_1{\cal U}\,+\,\cdots\,+X_n{\cal U}$, is a prime ideal
in $\cal U$, which is dense in $\cal M$ and distinct from $\cal
M$. \item[{\rm (ii)}] There are no characters on the quotient
algebra ${\cal U}/{\cal I}+(g - 1){\cal U}$. \item[{\rm (iii)}]
All characters on the commutative Fr\'{e}chet algebra $\cal U$ are
continuous.
\end{enumerate}
\end{cor}
\noindent {\it Remarks D.} 1. Esterle introduced notions of
Picard-Borel ideal and Picard-Borel algebra in ~\cite{20}. By
Corollary \ref{spectrum_test case}, it is clear that there does
not exist a maximal Picard-Borel ideal of $\mathcal V\,=\,{\cal
U}/{\cal I}$ distinct from ${\cal M}/{\cal I}$ of codimension $1$.
So all maximal Picard-Borel ideals are of infinite codimension.

\noindent 2. In view of ~\cite[Problem 6.5]{20}, the quotient
algebra $\cal V$ certainly does not possess the finite extension
property (see Definition 6.1, p. 576), if we show that the
converse of ~\cite[Corollary 3.4]{17} holds true (see comments
after proof of ~\cite[Theorem 6.6]{20}). However, we, here, take
an alternate approach to answer Esterle's question as follows.
Since $\cal I$ is a dense, prime maximal Picard-Borel ideal of
$\cal U$ of infinite codimension, by ~\cite[Theorem 5.9]{20}, we
have a finite family $(v_1, \dots, v_n)$ of elements of $\cal V$
with the following two properties: (1) $v_i\,=\,\pi_\infty(f_i)$,
where $f_i$ is a recalcitrant element in $\cal M\,\setminus\,{\cal
I}$ for $1\,\leq\,i\,\leq\,n$ and $\pi_\infty$ is the canonical
quotient mapping from $\cal U$ onto $\cal V$; and (2) $\sigma(v_1,
\dots, v_n)\,=\,\cc^n\,\neq\,\{\chi(v_1), \dots,
\chi(v_n)\}_{\chi\,\in\,M({\mathcal V})}$. We use these facts in

\noindent {\bf Proposition.} {\it The algebra $\cal V$ does not
possess the finite extension property.}

\noindent{\it Proof.} First, we remark that the algebra $\cal V$
has neither been topologized by Clayton in ~\cite{9} nor by
Esterle in ~\cite{20}. However, ${\cal V} = \bigcup_{n \geq
1}{\cal V}_{n}$ is, indeed, an ${\mathcal LF}-$algebra since it is
a strict inductive limit of an increasing sequence $({\cal V}_{n}
= {\cal U}/{\cal I}_n)$ of commutative unital Fr\'{e}chet algebras
as follows. $\cal V$ has the usual inductive topology w.r.t. the
sequence $({\cal V}_{n},\,\tau_n,\,\textrm{in}_n)$, inducing the
finest locally $m-$convex topology $\tau$ for which $\textrm{in}_n
: {\cal V}_n \rightarrow {\cal V}$ is continuous for each $n \in
\nn$ (for details on ${\mathcal LF}-$algebras and the inductive
topology, see ~\cite[\S 15]{29}). Note that ${\cal V} =
\lim\limits_{\longrightarrow}({\cal V}_{n},\,r_n)$, where, for
each $n \in \nn$, $r_n : {\cal V}_{n} \rightarrow {\cal V}_{n +
1}$ is the connecting mapping obtained by the restriction, and
since each ${\cal V}_n$ is semi-simple by ~\cite[Corollary
2.18]{20}, $\tau_n$ and $\tau_{n + 1}$ are equivalent on ${\cal
V}_{n}$. Thus $\cal V$ is a complete LMC-algebra, and by
~\cite[Theorem 2.2]{21}, $\cal V$ satisfies the Gel'fand-Mazur
theorem. By ~\cite[Proposition 15.9]{29} and ~\cite[Corollary
2.6]{20}, the maximal ideal space $M({\cal V})$ is $(\bigcup_{n
\geq 1}\ell^\infty_n) \bigcup \{\pi_0\}$, where $\pi_0$ is the
trivial character of $\cal V$ pertaining to the point evaluation
at zero. Interestingly, by ~\cite[Corollaries 10 and 11]{9} and
Corollary 3.4 above, $M({\cal V})$ reduces to $\{\pi_0\}$ (see,
also, ~\cite[\S 5, p. 570]{20}). So, by an explanation on p. 576
of ~\cite{20} and ~\cite[Corollary 11]{9}, $\cal V$ has the finite
extension property if and only if $\sigma(M) = \{\pi_{0}|_{M}\}\;
(M \subset {\cal V},\, M \neq \emptyset)$. However, as discussed
above, there is a finite family $(v_1, \dots, v_n)$ of elements of
$\cal V$ such that $\sigma(v_1, \dots, v_n) = \cc^n \neq
\{\pi_0(v_1), \dots, \pi_0(v_n)\}$. This shows that $\cal V$ does
not possess the finite extension property (it already does not
possess the countable extension property by ~\cite[Proposition
6.3]{20}). $\hfill \Box$

Esterle pointed out in abstract of ~\cite{20} new algebraic
obstructions to the construction of discontinuous characters on
$\cal U$ related to the Picard theorem, and relate to extension
properties of joint spectra of finite families of $\cal V$ a
question about iteration of Bieberbach mappings raised in
~\cite{17} (see ~\cite[Theorem 6.6]{20}). So he found some
difficulties while looking for a negative answer to Michael's
problem (see a paragraph preceding to ~\cite[\S 7]{20}). This was
also one of the starting points for author to work on positive
direction(s).

\noindent 3. In view of ~\cite[Corollary 2.20]{20} and Corollary
3.4 above, in presence of the continuum hypothesis (CH), the
quotient algebras $\cal V$ and ${\cal V}(g)\,=\,{\cal U}/{\cal
I}+(g - 1){\cal U}$ are not normable. It would be interesting to
check whether there exist models of set theory, including the
axiom of choice (AC), for which the algebra $\cal V$ is not
normable. However, one believes that there may exist models of set
theory, including the AC, in which $\cal V$ is normable. For this,
we work along the line of the result of Allan in
~\cite[Theorem]{3}, obtained by using the ZFC alone (see 5.
below), and so $\cal U$ is normable, since it is a FrAPS, by
~\cite[Theorem 10.1 (ii)]{15}. Now we use certain facts about the
algebras $\cal U$, ${\cal U}/{\cal I}_m$ (resp., ${\cal U}/{\cal
I}_m + (g - 1){\cal U}$) for $m\,\in\,\nn$ and $\cal V$ (resp.,
${\cal V}(g)$) from ~\cite[Lemma 2.17, Corollary 2.18 and Theorem
2.19]{20} in order to show that the integral domain $\cal V$
(resp., ${\cal V}(g)$) can (algebraically) be embedded in ${\cal
F}$ since ${\cal U}/{\cal I}_n$ is isomorphic to $\cal U$ for each
$n$. Thus, in presence of ZFC alone, the quotient algebras $\cal
V$ and ${\cal V}(g)$ are normable ~\cite[Theorem 2]{4}, and we
have partially answered a question of Esterle from ~\cite[p.
554]{20}.

\noindent 4. By ~\cite{9}, $(\ell^\infty, w^*)$ is not a first
countable (and hence, not metrizable) topological space; in fact,
it is not even a $k$-space, but it is a Lindel\"{o}f and
completely regular space. Hence the algebra $C(\ell^\infty)$ of
continuous functions on $\ell^\infty$ is not a Fr\'{e}chet algebra
by ~\cite[Chapter 3, Theorem]{23}; it is just a commutative,
non-complete LMC algebra ~\cite[Proposition 4.10.20]{12}. All the
characters on $C(\ell^\infty)$ are bounded, since $(\ell^\infty,
w^*)$ is replete ~\cite[Corollary 4.10.23]{12}; not only this, but
$C(\ell^\infty)$ is functionally continuous by ~\cite[Theorem
4.10.24]{12}. However, the subalgebra $\textrm{Hol}(\ell^\infty)$
is a Fr\'{e}chet algebra ~\cite[Proposition 3]{9}, which is also
functionally continuous and functionally bounded by Corollaries
3.2 and 3.3 above. Thus one sees the existence of functionally
continuous, functionally bounded, non-complete, non-metrizable LMC
algebra. The author does not know any such examples in the
literature. The commutative, complete, unital LMC-algebra
$\Upsilon(\ell^\infty)\,=\,O(\ell^\infty)$ of germs of analytic
functions in a neighbourhood of $\ell^\infty$, is another
interesting algebra which is functionally continuous (and hence,
functionally bounded) ~\cite[\S 4]{10}.

\noindent 5. Allan exhibited a discontinuous homomorphism between
two commutative, unital Fr\'{e}chet algebras having certain
properties ~\cite[Theorem 8]{4}. In the construction of
discontinuous homomorphism, he used a continuous homomorphism from
$A$ into $\F$, induced by a natural, continuous higher point
derivation $(d_n)$ of infinite order on the Fr\'{e}chet algebra
$\textrm{Hol}(\cc)$ of entire functions at a continuous character
$d_0$. He used only the ZFC axioms (and not the CH). Similarly,
Read used the ZFC axioms alone to give an inequivalent Fr\'{e}chet
algebra topology on the algebra ${\cal F}_\infty$, in order to
show that the Singer-Wermer conjecture fails in the Fr\'{e}chet
case ~\cite{37}. Thus we, here, establish $NDH_{F,\,\cc}$, where
$F$ is any Fr\'{e}chet algebra (see ~\cite[\S 9]{11} for details).
In ~\cite{15}, Dales, author and Read also used only the ZFC
axioms to establish several important results in automatic
continuity theory, including an affirmative answer to the
Dales-McClure problem, and showing that the test algebra $\cal U$
for Michael's problem is, indeed, a FrAPS in $\F$. We remark that
Michael's problem arises in the ZFC.

\noindent 6. It is clear that the method used in our approach
works only for commutative Fr\'{e}chet algebras. However, the
non-commutative analogues of Michael's problems also exist, and
Dixon and Esterle discussed the test case ${\cal U}_{nc}$ (the
non-commutative analogue of $\cal U$) in ~\cite{17}. In another
paper, we exclusively discuss the second approach, based on the
Read's method of generating other inequivalent Fr\'{e}chet
topology on the test case ${\cal U}_{nc}$, to affirmatively solve
Michael's problems in the non-commutative case. As a welcome
bonus, we also establish some significant, long-standing problems
in automatic continuity theory, including an affirmative answer to
Michael's problem for a more general class of $(F)-$algebras by
the third approach.

\noindent 7. Until now, we know that $M(A)$ is dense in $S(A)$,
where $S(A)$ (resp., $M(A)$) is the space of all (continuous)
characters on $A$ w.r.t. the Gel'fand topology ~\cite[Lemma
10.1.1]{23}. By Corollary 3.2, we have $M(A)\,=\,S(A)$.

\noindent 8. We have throughout assumed that our algebras have an
identity. It is a triviality to show that if a commutative
Fr\'{e}chet algebra has a discontinuous character, then so does
the algebra obtained by adjoining an identity
(see~\cite[Proposition 2.1]{17} or ~\cite[Remarks 4.5 (2)]{10}).
Thus, as far as Michael's problem is concerned, no loss of
generality is involved in our assumption.

\noindent 9. Dales and McClure used the topological version of the
tensor algebra over a Banach space to obtain a somewhat negative
result in automatic continuity theory whereas we use the same
technique to answer Michael's problem affirmatively. Also, along
the same lines, recall our comments, given in Remarks B. 3, on the
use of the Dales-author-Read's method in the proof of
~\cite[Theorem 10.1]{15} as well as in the proof of the main
theorem. Thus our approach confirms the impression that one may
not expect the phenomena to hold in the Fr\'{e}chet algebra theory
that hold in the Banach algebra theory. However, we are lucky to
answer Michael's problem in our favor ~\cite[Conclusion]{37}.
Hence, the situation on Fr\'{e}chet algebras is markedly different
from that on Banach algebras, and that a structure theory for
Fr\'{e}chet algebras behaves in a very distinctive manner from
Banach algebras.
\section{Applications to automatic continuity theory}
\noindent {\bf 4.1. Some remarks in automatic continuity theory.}
As promised, we have affirmatively answered Michael's problem
within automatic continuity theory. In the past, there were some
significant works in this theory, which relate to our past (as
explained in Introduction) and present works in some or the other
way and giving applications within this theory or in the theory of
commutative rings and algebras as follows. Johnson established the
uniqueness of the complete norm for BAPS ~\cite{24}. This result
was extended by Loy to FrAPS satisfying the equicontinuity
condition (E) ~\cite{27} (cf. ~\cite{31, 33}; also, for certain
FrAPS in $\F_\infty$ ~\cite{33}). Thomas showed that the image of
a derivation on a Banach algebra is contained in the radical,
solving the Singer-Wermer conjecture affirmatively ~\cite{40}.
However, Read showed that this conjecture fails in the Fr\'{e}chet
case ~\cite{37} (cf. ~\cite{35}). Vogt gave uncountably many
inequivalent Fr\'{e}chet space topologies to spaces of holomorphic
functions ~\cite{42}. However, all these spaces are semi-simple
Fr\'{e}chet algebras with a unique Fr\'{e}chet topology ~\cite{8,
31, 15}. Loy gave a simple method to construct commutative Banach
algebras with inequivalent complete norms by using the
discontinuity of derivations ~\cite{28} (cf. ~\cite{34}).\\

\noindent {\it Remarks E.} 1. Our first remark is on the two
problems of Michael ~\cite[\S 12]{29}, posed exclusively for
commutative algebras, as follows. Dales posed a question about the
continuity of characters on a commutative locally convex
$(F)$-algebras ~\cite[Question 3 (ii)]{11}. Dixon discussed the
non-equivalence of Michael's two problems for more general
commutative locally convex $(F)$-algebras and commutative,
complete locally convex algebras (the reason is that he could not
extend ~\cite[Theorem 1]{18} in this case). Also, he showed that
there is a commutative, complete locally convex algebra, with
jointly continuous multiplication and with an unbounded character
~\cite[\S 4, Questions (3) and (4), and Theorem 4.2]{16}. Thus,
every commutative, complete LMC-algebra is functionally bounded by
Corollary 3.3, but this is not true for every commutative,
complete locally convex algebra.

\noindent 2. Carpenter ~\cite{8} showed that every commutative
semi-simple Fr\'{e}chet algebra $A$ admits a unique Fr\'{e}chet
topology. The proof was divided into four parts, but if characters
were continuous on $A$, then the proof could be derived from the
third part only (see Corollary 4.2 below).\\

\noindent {\bf 4.2. Another equivalent form of Michael's problem.}
First, we see the following easy extension of the \v{S}ilov's
theorem ~\cite[Theorem 2.3.3]{12}. By ~\cite[Proposition
8.1.2]{23}, $\textrm{rad}\,A = \bigcap_{\psi \in
S(A)}\,\textrm{ker}\,\psi\,=\,\bigcap_{\phi \in
M(A)}\,\textrm{ker}\,\phi$, below.

\noindent {\bf Theorem.} {\it Let $B$ be a Fr\'{e}chet algebra
such that all characters of $B$ are continuous (i.e.,
$S(B)\,=\,M(B)$), let $A$ be a commutative Fr\'{e}chet algebra,
with the space of characters $S(A)$, and let $\theta : B
\rightarrow A$ be a homomorphism. Then ${\cal S}(\theta) \subset
\textrm{rad}\,A$.}

\noindent {\it Proof.} Suppose that $b_n\,\rightarrow\,0$ in $B$
and $\theta(b_n)\,\rightarrow\,a$ in $A$, and take
$\phi\,\in\,S(A)$. Then
$\phi\,\circ\,\theta\,\in\,M(B)\,\bigcup\,\{0\}$, and so
$\phi\,\circ\,\theta$ is continuous. Thus
$$\phi(a)\,=\,\lim_{n\,\rightarrow\,\infty}(\phi\,\circ\,\theta)(b_n)\,=\,0.$$
This shows that
$a\,\in\,\bigcap\{\textrm{ker}\,\phi\,:\,\phi\,\in\,M(A)\}\,=\,\textrm{rad}\,A$.
So ${\cal S}(\theta) \subset \textrm{rad}\,A$. $\hfill \Box$

As a special case, we have the following

\noindent {\bf Corollary.} {\it Let $B,\, A$ and $\theta$ be as in
the above theorem, with $A$ semi-simple. Then $\theta$ is
automatically continuous.}

\noindent {\it Proof.} Since $A$ is semi-simple,
$\textrm{rad}\,A\,=\,\{0\}$, and so is ${\cal S}(\theta)$. Hence,
by the closed graph theorem for Fr\'{e}chet spaces, $\theta$ is
continuous. $\hfill \Box$

It is necessary to include the hypothesis that all characters of
$B$ are continuous in the above corollary because it may be that
there is a discontinuous character $\phi$ on a non-commutative
Fr\'{e}chet algebra $B$, and this would give a discontinuous
homomorphism $b\,\mapsto\,\phi(b)1,\;B\,\rightarrow\,A$. However,
having solved Michael's problem affirmatively in \S 3, below, we
solve another long-standing open problem in the theory of
Fr\'{e}chet algebras, equivalent to Michael's problem ~\cite[p.
143]{11} and ~\cite[p. 134]{31}.
\begin{thm} \label{continuity_ssfa} Let $A$ be a commutative
semi-simple Fr\'{e}chet algebra and let $B$ be a commutative
Fr\'{e}chet algebra. Let $\theta : B \rightarrow A$ be a
homomorphism. Then $\theta$ is automatically continuous.
\end{thm}
\noindent{\it Proof.} First, we assume, w.l.o.g., that $B$ is
unital, because if $B$ is not unital, then we can adjoin the
identity $e$ and we can extend $\theta$ on $B_{e}$ by taking
$\theta(e) = e^{'}$ (if $A$ is not unital, then we can also adjoin
the identity $e^{'}$ to $A$). Obviously, $\theta$ is continuous if
and only if an extension of $\theta$ is continuous.

By the closed graph theorem for Fr\'{e}chet spaces, it suffices to
show that, if $b_n \rightarrow b$ in $B$, and if $\theta(b_n)
\rightarrow a$ in $A$, then $\theta(b) = a$. But for $\phi \in
M(A)$, $\phi \circ \theta \in M(B)$, and so $$\phi(a) =
\lim\phi(a_n) = \lim(\phi \circ \theta)(b_n) = (\phi\,\circ
\theta)(b).$$ Thus, $\phi(a - \theta(b)) = 0\;(\phi \in M(B))$. We
now use the fact that $A$ is both commutative and semi-simple
algebra to deduce that $a = \theta(b)$, as required.$\hfill \Box$

As an application to the above theorem, we can drastically shorten
the proof of Carpenter ~\cite{8}, establishing the uniqueness of
the Fr\'{e}chet topology of a commutative, semi-simple Fr\'{e}chet
algebra in the following
\begin{cor} \label{uniqueness_Carpenter} Every commutative,
semi-simple Fr\'{e}chet algebra admits a unique Fr\'{e}chet
algebra topology.
\end{cor}
\noindent{\it Proof.} Let $A$ be a commutative, semi-simple
Fr\'{e}chet algebra with respect to another Fr\'{e}chet topology
$\sigma$, distinct from the Fr\'{e}chet topology $\tau$. Consider
the identity mapping from $(A, \tau)$ into $(A, \sigma)$.
Evidently it is a continuous homomorphism by Theorem 4.1. By the
open mapping theorem for Fr\'{e}chet spaces, it is a linear,
homeomorphism, and so $\sigma\,=\,\tau$. $\hfill \Box$

\noindent 3. Forster did not know whether every homomorphism
$\theta\,:\,A\,\rightarrow\,B$ is continuous, where $A$ and $B$
are Stein algebras, with $\textrm{dim}M(A)\,=\,\infty$
(~\cite[REMARK]{22}). We settle this problem in the affirmative as
follows. We first see that Theorem 5 of ~\cite{22} holds true in
the following
\begin{cor} \label{Thm_Forster} Let $A$ and $B$ be two Stein
algebras, with $\textrm{dim}M(A)\,=\,\infty$, and $(X_B,\,\cal B)$
a reduced Stein space. Then every homomorphism
$\theta\,:\,A\,\rightarrow\,B$ is continuous.
\end{cor}
\noindent{\it Proof.} By the definition of a Stein algebra, $B
\cong \textrm{Hol}(X_B, \cal B)$, $(X_B, \cal B)$ a Stein space.
Since $(X_B, \cal B)$ is reduced (i.e., a complex space in the
sense of Serre ~\cite{38}), $B$ is a semi-simple Fr\'{e}chet
algebra. By Theorem 4.1, $\theta$ is continuous.$\hfill \Box$

The above corollary is, indeed, a version of Theorem 4.1 in the
class of Stein algebras. If we replace a purely algebraic
condition on $B$ in the above corollary by $\theta$ being an
epimorphism, then we have the following
\begin{cor} \label{Cor_Forster} Let $A$ and $B$ be two Stein
algebras, with $\textrm{dim}M(A)\,=\,\infty$. Then every
epimorphism $\theta\,:\,A\,\rightarrow\,B$ is continuous.
\end{cor}
\noindent{\it Proof.} Let $M$ be an arbitrary finitely generated
maximal ideal of $B$. Then $M$ is closed by ~\cite[Theorem 2]{22}
and $B/M \cong \cc$. Hence $\theta^{-1}(M)$ is also a maximal
ideal of $A$ such that $A/\theta^{-1}(M) \cong \cc$. So,
$\theta^{-1}(M)$ is closed being the kernel of the quotient map
from $A$ onto $A/\theta^{-1}(M)$. Again, by ~\cite[Theorem 2]{22},
$\theta^{-1}(M)$ is finitely generated. Hence, by ~\cite[Theorem
3]{22}, $\theta$ is continuous. $\hfill \Box$

A somewhat more elaborate version of the same idea enables us to
drop the surjectivity condition on $\theta$ in order to {\it
completely} settle the Forster's problem in the following
\begin{cor} \label{Cor_final} Let $A$ and $B$ be two Stein
algebras, with $\textrm{dim}M(A)\,=\,\infty$. Then every
homomorphism $\theta\,:\,A\,\rightarrow\,B$ is continuous.
\end{cor}
\noindent{\it Proof.} Suppose that $\theta$ is a homomorphism.
Following the proof of Corollary 4.4, $\theta^{-1}(M)$ is just an
{\it ideal} of $A$ such that $A/\theta^{-1}(M) \cong \cc$. So,
$\theta^{-1}(M)$ is closed being the kernel of the quotient map
from $A$ onto $A/\theta^{-1}(M)$. Since the ideal $\theta^{-1}(M)$
satisfies the hypothesis of Theorem 4 of ~\cite{22}, it is also
finitely generated. Thus, for every finitely generated maximal
ideal $M$ in $B$, the ideal $\theta^{-1}(M)$ is finitely generated
in $A$. By ~\cite[Theorem 3]{22}, $\theta$ is continuous. $\hfill
\Box$

We remark that we have exposed the elementary nature of Forster's
theorems (~\cite[Theorems 3 and 4]{22}) by using some elementary
facts from abstract algebra and functional analysis. We also
remark that Theorem 4.1 above does not hold true for {\it all}
Fr\'{e}chet algebras as happened with the class of Stein algebras
in Corollary 4.5 (e.g., ~\cite[Theorem 11.2]{15}, Theorem 1.1
above, and the algebra $\F_\infty$, with countably many
inequivalent Fr\'{e}chet topologies ~\cite[pp. 125-126]{35}, and
so, the identity map is not a topological isomorphism).

\noindent 4. In ~\cite[Theorems 4.1]{31, 33}, we can now drop the
condition ``the range of $\theta$ is not one-dimensional", because
there are no discontinuous characters on the domain algebra $B$
which would have given the discontinuous homomorphism
$b\,\mapsto\,\phi(b)1, \;B\,\rightarrow\,A$.\\

\noindent {\bf 4.3. Completing the circle of ideas.} In
~\cite[Theorem 12.3]{15}, it was shown that a Banach algebra
$\ell^1(S)\,\cong\,\widehat{\bigvee}_{\{1\}}E$, is such that
$\cc[X_1,\,X_2] \subset \ell^1(S) \subset \F_{2}$, but the
embedding $(\ell^1(S),\,\|\cdot\|)\rightarrow (\F_{2},\,\tau_c)$
is not continuous. We did not know whether there is a non-Banach
Fr\'{e}chet algebra with these properties. We show that the test
case $\cal U$ is such an example in the following
\begin{thm} \label{Thm 12.3_DPR} The test case $({\cal U},
(q_m))$ for Michael's problem is such that $\cc[X_1, X_2] \subset
{\cal U} \subset \F_{2}$, with the embedding $({\cal U}, (q_m))
\rightarrow (\F_{2}, \tau_c)$ discontinuous.
\end{thm}
\noindent{\it Proof.} The proof is the same as that of
~\cite[Theorem 12.3]{15}. So we write $\F_2$ as $\cc[[X, Y]]$ and
reserve $X_i$ for elements of $\cal U$. We regard
$\F\,=\,\cc[[X]]$ as a subalgebra of $\cc[[X, Y]]$; the obvious
quotient map from $\cc[[X, Y]]$ obtained by setting $Y\,=\,0$ is
denoted by
$$\pi\,:\,\sum_{i, j = 0}^\infty \alpha_{i, j}X^iY^j\,\mapsto\,\sum_{i = 0}^\infty\alpha_{i, 0}X^i, \;\;\F_2\,\rightarrow\,\F.$$
By ~\cite[Theorem 10.1 (ii)]{15}, there is a continuous, unital
embedding $\theta\,:\,{\cal U}\,\rightarrow\,\F$ such that
$\cc[X]\,\subset\,\theta({\cal U})$. We set
$f_i\,=\,\theta(X_i)\;(i\,\in\,\nn)$, and as in ~\cite[Theorem
10.1]{15}, we may suppose that $f_1\,=\,X\,\in\,\F$. We denote by
${\cal U}^{(1)}$ the closed linear subspace of $\cal U$ spanned by
the elements $X_i$ for $i\,\in\,\nn$, so that
\begin{equation} \label{U1} {\cal U}^{(1)}=\bigcap_{m
\in \nn}{\cal
U}_{m}^{(1)}=\{\sum_{i=1}^\infty\alpha_iX_i:\sum_{i=1}^\infty|\alpha_i|m
< \infty,\;\textrm{for all}\,m \in \nn\},\;{\cal U}_{1}^{(1)}
\cong \ell^1(\zz^+),
\end{equation}
and so this non-Banach Fr\'{e}chet space is not isometrically
isomorphic to $\ell^1(\zz^+)$. But we can still choose a non-zero,
discontinuous linear functional $\lambda$ on ${\cal U}^{(1)}$ (as
the element $\sum_{i=2}^\infty\frac{X_i}{i^2}\,\in\,{\cal
U}^{(1)}$), and then define a linear map
\begin{equation} \label{linear-psi}
\psi\,:\,{\cal
U}^{(1)}\,\rightarrow\,\F_2,\;u\,\mapsto\,\theta(u)\,+\,\lambda(u)Y.
\end{equation}

Our main {\it claim} is that $\psi$ can be extended to a
homomorphism $\Psi{'}\,:\,{\cal U}\,\rightarrow\,\F_2$ such that
$\pi\,\circ\,\Psi{'}\,=\,\theta$. To establish this claim, we
shall require the following slightly more general theorem, in
which we maintain the above notations. Further, we write ${\cal
M}_2$ for the unique maximal ideal of $\F_2$. We omit the proof of
the theorem by recalling that ${\cal U}^{(1)}$ is a subspace of
${\cal U}_1^{(1)}\,=\,\ell^1(\zz^+)$, both spanned by $X_i$ for
$i\,\in\,\nn$, and $\cal U$ is a subalgebra of ${\cal
U}_1\,=\,\ell^1(S)$, so the proof of ~\cite[Theorem 12.4]{15}
works in our case as well.
\begin{thm} \label{Thm 12.4_DPR} Let $\beta\,:\,{\cal
U}^{(1)}\,\rightarrow\,{\cal M}_2$ be a linear map such that
$\pi\,\circ\,\beta\,:\,{\cal U}^{(1)}\,\rightarrow\,\F$ is
continuous. Then there is a unital homomorphism
$\bar{\beta}\,:\,{\cal U}\,\rightarrow\,\F_2$, extending $\beta$,
such that $\pi\,\circ\,\bar{\beta}\,:\,{\cal U}\,\rightarrow\,\F$
is continuous.
\end{thm}
In fact, a slight digression of the proof of Theorem 4.7 (which is
analogous to the proof of ~\cite[Theorem 12.4]{15}, and a similar
proof is discussed below) enables us to give the following
``natural" variation of Theorem 4.7 in its fuller generality. Both
Theorems 4.7 and 4.8 are of independent interest in view of our
approach to Michael's problem.
\begin{thm} \label{Thm 5.2_Patel} Let $\beta\,:\,{\cal
U}^{(1)}\,\rightarrow\,\F$ be a continuous linear map. Then there
is a continuous, unital homomorphism $\bar{\beta}\,:\,{\cal
U}\,\rightarrow\,\F$, extending $\beta$.
\end{thm}
As opposite to Theorem 4.6 above, we provide a much shorter and
elegant way to embed $\cal U$ into $\F$ in the following corollary
(cf. ~\cite[Theorem 10.1]{15}).
\begin{cor} \label{Thm 5.3_Patel} The Fr\'{e}chet algebra $\cal U$ is (isometrically isomorphic to)
a Fr\'{e}chet algebra of power series.
\end{cor}
\noindent{\it Proof.} First, we remark that we again require to
take further digression in the proof of Theorem 4.8 as follows.
For each $i\,\in\,\zz^+$, we take $\beta_{(i)}$ to be the usual
coordinate linear functional on ${\cal U}^{(1)}$. We then define a
mapping $$\beta\,:\,{\cal U}^{(1)}\,\rightarrow\,\F \textrm{such
that} \beta(f)\,=\,\sum_{i=0}^\infty\,\beta_{(i)}(f)X^i.$$
Clearly, such a $\beta$ is continuous and injective, and
$\beta(X_1)\,=\,X$. We extend each $\beta_{(i)}$ to a linear
functional $\beta_{(i)}$ on $\F_\infty^{(1)}$. Next we define a
linear functional $\beta_{(i)}^{(n)}$ on $\F_\infty^{(n)}$ for
each $n\,\in\,\nn$ by the following formula:
\begin{equation} \label{formula-1} \beta_{(i)}^{(n)}(f)\,=\,\sum\{(\beta_{i^{(1)}}\,\otimes\,\cdots\,\otimes\,\beta_{i^{(n)}})(\epsilon_n(f))\}\;\;(f\,\in\,\F_\infty^{(n)}),
\end{equation}
where the sum is taken over all $n$-tuples $(i) =
(i^{(1)},\,\dots\,i^{(n)}) \in (\zz^+)^n$ such that\linebreak
$i^{(1)} + 2i^{(2)} + \cdots + ni^{(n)} = \omega((i))$, a weighted
order of $(i)$ ~\cite[pp. 137-139]{15}. We now {\it claim} that
the map $\bar{\beta}\,:\,{\cal U}\,\rightarrow\,\F$, defined for
$f\,\in\,\F_\infty$ by the formula
\begin{equation} \label{formula-2} \bar{\beta}(f)\,=\,\sum_{k=0}^{\infty}\{(\sum\{\beta_{(i)}^{(n)}(f^{(n)}\,:\,n\,\in\,\nn_i\})X^i\,:\,\omega((i))\,=\,k\},
\end{equation}
where we set $\bar{\beta}^{(0)}(f)\,=\,f(0)1$, is a unital
homomorphism $\bar{\beta}\,:\,{\cal U}\,\rightarrow\,\F$
satisfying the stated conditions. Follow the proof of
~\cite[Theorem 12.4]{15} to show that $\bar{\beta}$ is a unital
homomorphism that extends $\beta$. Again follow the proof of
~\cite[Theorem 12.4]{15} to show that $\bar{\beta}$ is continuous,
where we use the continuity of coordinate linear functionals
$\beta_{(i)}$ and the fact that the 'tensor product by rows'
~\cite[Definition 1.9]{37} agrees with the usual tensor product
when the linear functionals are continuous. Finally, we need to
show that $\bar{\beta}$ is injective. Now, follow the proof of
~\cite[Theorem 9.1]{15} for this to derive. Thus the algebra $\cal
U$ is continuously embedded in $\F$. In this case,
$\bar{\beta}({\cal U})$ is a FrAPS in $\F$, w.r.t. the metric
transferred from $\cal U$, and so $\cal U$ is isometrically
isomorphic to a FrAPS. $\hfill \Box$

If we replace $\F$ by $\cc$ in the codomain in Theorem 4.8, then a
similar argument in the proof above also enables us to extend a
continuous linear functional $\beta\,=\,\beta_{(0)}$ on ${\cal
U}^{(1)}$ (resp., on $\ell^1(\zz^+)$) to a continuous character
$\bar{\beta}$ on $\cal U$ (resp., on the Banach algebra ${\cal
U}_1$) as follows. In this case, we then define a mapping
$\beta\,:\,{\cal U}^{(1)}\,\rightarrow\,\cc$ such that
$\beta(f)\,=\,\beta_{(0)}(f)$. Clearly, such a $\beta$ is
continuous and $\beta(X_1)\,=\,0$. We extend $\beta_{(0)}$ to a
linear functional $\beta_{(0)}$ on $\F_\infty^{(1)}$. Next we
define a linear functional $\beta_{(0)}^{(n)}$ on
$\F_\infty^{(n)}$ for each $n\,\in\,\nn$ by the following formula:
\begin{equation} \label{formula-1} \beta_{(0)}^{(n)}(f)\,=\,(\beta_{(0)}\,\otimes\,\cdots\,\otimes\,\beta_{(0)})(\epsilon_n(f))\;\;(f\,\in\,\F_\infty^{(n)}),
\end{equation}
where the $n$-tuple $(0) = (0,\,\dots,\,0) \in (\zz^+)^n$. We now
claim that the map $\bar{\beta}\,:\,{\cal U}\,\rightarrow\,\cc$,
defined for $f\,\in\,\F_\infty$ by the formula (4.4) above, where
we set $\bar{\beta}^{(0)}(f)\,=\,f(0)1$, is a character
$\bar{\beta}\,:\,{\cal U}\,\rightarrow\,\cc$ satisfying the stated
conditions. Follow the proof of ~\cite[Theorem 12.4]{15} to show
that $\bar{\beta}$ is a character that extends $\beta$. Again
follow the proof of ~\cite[Theorem 12.4]{15} to show that
$\bar{\beta}$ is continuous; in the process, one can either
consider `tensor product by rows', or usual tensor product of a
continuous linear functional $\beta$, it makes no difference.

The main point should be emphasized here: if one starts with a
discontinuous character on $\cal U$, then one obtains a
discontinuous linear functional on ${\cal U}^{(1)}$. If one starts
with a continuous linear functional on ${\cal U}^{(1)}$, then one
obtains a continuous character on $\cal U$. In particular, if one
starts with the usual coordinate linear functional
$\beta\,=\,\beta_{(0)}$ on ${\cal U}^{(1)}$, then, by Corollary
4.9 (or, ~\cite[Theorem 10.1 (ii)]{15}) and the above arguments,
one obtains the character $\pi_0$ on $\cal U$, which is continuous
~\cite[Corollary 11.5]{15}.

Further, we recall some interesting facts. By Theorem 1.1,
$\F_\infty$ cannot be embedded into $\F$ either algebraically or
topologically whereas, by Corollary 4.9 (or, ~\cite[Thoerem
10.1]{15}), the test algebra $\cal U$ (resp., the Banach algebra
${\cal U}_1\,=\,\ell^1(S)$), which is a FrAPS (resp., BAPS) in
$\F_\infty$, is (isometrically isomorphic to) a FrAPS (resp.,
BAPS) in $\F$. However, when we replace $\F$ by $\F_2$ in the
codomain in these results, then the scenario is totally opposite.
That is, by ~\cite[Theorem 9.1]{15}, $\F_\infty$ can be embedded
into $\F_2$ whereas, by Theorem 4.6 (resp., ~\cite[Theorem
12.3]{15}), the test algebra $\cal U$ (resp., the Banach algebra
${\cal U}_1$) is not a FrAPS (resp., a BAPS) in $\F_2$.

We can now complete the proof of Theorem 4.6.

Theorem 4.7 shows that $\psi$ can be extended to a homomorphism
$\Psi{'}\,:\,{\cal U}\,\rightarrow\,\F_2$ such that
$\pi\,\circ\,\Psi{'}\,=\,\theta$. The map $\Psi{'}$ is an
embedding because $\theta$ is injective, and it is manifest that
$\Psi{'}$ is discontinuous. It is clear that $\Psi{'}({\cal U})$
contains $\F$ and the element $\Psi{'}(\sum_{i = 2}^\infty
X_i/i^2)$, which, by a remark in the proof of ~\cite[Theorem
10.1]{15}, has the form $X^2f + Y$ for some $f\,\in\,\F$. By
~\cite[Lemma 1.2]{15}, there is a continuous, unital automorphism
$\chi$ of $\F_2$ such that $\chi(X) = X$ and $\chi(X^2f + Y) = Y$.
So $\chi \circ \Psi{'} : {\cal U} \rightarrow \F_2$ is a
discontinuous embedding whose rang contains $\cc[X, Y]$, as
required for the proof of the theorem (where we identify $\cal U$
with its image in $\F_2$ under $\chi \circ \Psi{'}$). $\hfill
\Box$\\

\noindent {\bf Acknowledgements.} We take this opportunity to
remember and thank significant mathematicians, who contributed to
Michael's problem in the past and who, unfortunately, passed away
before this work comes in existence, especially, Stanis\l aw Mazur
(1981), Richard Arens (2000), Barry Johnson (2002), Graham Allan
(2007), Ernest Michael (2013), Charles Read (2015), Jorge Mujica
and Marc Thomas (2017), Berit Stensones and Garth Dales (2022),
Jean Esterle (2023) and Wieslaw \.Zelazko (2025).

\end{document}